\documentclass[12pt,leqno]{amsart}
\usepackage{amsmath,amssymb,amsfonts, amscd,   
pifont, amsthm,  amsxtra, latexsym, 
}
\usepackage[all]{xy}
\usepackage{here}
\usepackage{graphicx} 
\usepackage{amsmath,amssymb,amscd} 
\usepackage{extarrows} 
\usepackage{colortbl} 
\usepackage{pstricks,multido} 
\usepackage{pstricks,pst-arrow} 
\usepackage{pst-plot} 
\usepackage[colorlinks=true, linkcolor=blue, citecolor=blue, urlcolor=blue]{hyperref}

\title{Proper Actions and Representation Theory}
\author{Toshiyuki Kobayashi}

\numberwithin{equation}{section} %
\theoremstyle{plain} %
  \newtheorem{theorem}{Theorem}[section] %
  \newtheorem{proposition}[theorem]{Proposition}%
  \newtheorem{lemma}[theorem]{Lemma}%
  \newtheorem{corollary}[theorem]{Corollary}%
  
  \newtheorem{claim}[theorem]{Claim}
  \newtheorem{conjecture}[theorem]{Conjecture}
\theoremstyle{definition} %
  \newtheorem{definition}[theorem]{Definition}%
  \newtheorem{example}[theorem]{Example}%
  \newtheorem{problem}[theorem]{Problem}
  
  \newtheorem{def-prop}[theorem]{Definition-Proposition}
  \newtheorem{definition-lemma}[theorem]{Definition--Lemma}
  \newtheorem{question}[theorem]{Question}

\theoremstyle{remark} %
  \newtheorem{remark}[theorem]{Remark}%
\newenvironment{theorembis}[1]
  {\renewcommand{\thetheorem}{\ref{#1}$'$}%
   \addtocounter{theorem}{-1}%
   \begin{theorem}}
  {\end{theorem}}

\newcommand{\RED}[1]{\textcolor{red}{#1}}
\newcommand{\BLUE}[1]{\textcolor{blue}{#1}}

\begin{document}

\begin{abstract}
This exposition presents recent developments on proper actions,
 highlighting their connections to representation theory. It begins with geometric aspects,
 including criteria for the properness of homogeneous spaces in the setting of reductive groups.  
We then explore the interplay between the properness of group actions and the discrete decomposability of unitary representations realized on function spaces. 
Furthermore, 
two contrasting new approaches to quantifying proper actions are examined:
 one based on the notion of sharpness, which measures how strongly a given action satisfies properness;
 and another based on dynamical volume estimates,
 which measure deviations from properness.  
The latter quantitative estimates have proven especially fruitful in establishing temperedness criterion for regular unitary representations on $G$-spaces. 
Throughout, 
 key concepts are illustrated with concrete geometric and representation-theoretic examples.
\end{abstract}

\maketitle

\setcounter{tocdepth}{1}
\tableofcontents

\section{Introduction}
The actions of non-compact groups
 on manifolds
 can exhibit highly non-trivial
 and \lq\lq{wild behavior}\rq\rq.  
The notion of proper actions, 
 introduced by Palais \cite{P61}, 
abstracts and formalises the favorable features
 characteristic of actions of compact groups.  
A prototypical example
 of a proper action is the action of the fundamental group $\Gamma$
 of a manifold 
 on its universal covering space
 via deck transformations.

On the other hand, 
 when $X$ is a Riemannian manifold
 on which a discrete group acts freely and by isometries, 
 the action is automatically properly discontinuous
 (Proposition \ref{prop:isomR}).  
The quotient space
 $X_{\Gamma}=\Gamma \backslash X$ 
 inherits a natural Riemannian structure from $X$
 via the covering $X \to \Gamma \backslash X$, 
 thereby becoming a Riemannian manifold.  
In this setting, 
 one may regard $\Gamma$ as governing the global structure
 of the quotient manifold $\Gamma \backslash X$, 
 while
 the original manifold $X$ determines its local structure.

However, 
 in more general settings---such as when the Riemannian structure
 is replaced with a pseudo-Riemannian one
 (allowing indefinite metric)---the situation is significantly different:
 free actions
 by discrete groups of isometries often fail to be properly discontinuous
 ({\it{e.g.}}, Example \ref{ex:Lorentz_isometry}).

In the study of local-to-global phenomena
 beyond the Riemannian setting, 
 understanding proper actions
 (or properly discontinuous actions)
 is therefore crucial.

In this paper, 
 we examine recent progress
 concerning proper actions from both geometric and representation-theoretic perspectives.

\vskip 2pc
The exposition begins with the topological and geometric framework 
 related to group actions
 by using binary relations
 $\pitchfork$ and $\sim$ on the power set of $G$.  
Sections \ref{sec:discgp} and \ref{sec:proper} present 
 criteria for the properness of group actions
 on homogeneous spaces.  
Topics include Lipsman's conjecture for nilmanifolds
 (Section \ref{subsec:Lipsman})
 and the properness criterion
 (Theorem \ref{thm:proper96})
 in the reductive case, 
 and several subtle examples
 that illustrate these results.

In Section \ref{sec:cocompact}, 
 we give a brief overview 
 of recent developments
 concerning cocompact discontinuous groups
 for reductive homogeneous spaces $G/H$.  
Inspired by Mackey's philosophy---originally developed
 for unitary representations---we also mention a topological analogue
 involving the {\emph{tangential homogeneous space}}
 $G_{\theta}/H_{\theta}$, 
 which arises from their associated Cartan motion groups.

The conceptual link between {\emph{properness}} in topological group actions
 and {\emph{discrete decomposability}}
 in unitary representation theory---traditionally seen as unrelated domains 
 that have been developed through different methods 
 and perspectives---was first proposed in  \cite[Sect.\ 3]{K-Hayashibara97}.  
This work introduced the previously unexplored idea 
 that non-compact subgroups may exhibit
 {\emph{compact-like behaviour}}.  
Subsequent developments, 
 particularly those involving spectral analysis
 on locally pseudo-Riemannian symmetric spaces
 $\Gamma \backslash G/H$
 ({\it{e.g.}}, \cite{KaK25}), 
 have further deepened this perspective.  
In Section \ref{sec:proper_unitary}, 
 we investigate the interplay
 between the properness of group actions and the discrete decomposability 
 of unitary representations realized on function spaces.

In Section \ref{sec:quantify}, 
 we discuss two contrasting approaches to quantifying the properness
 of group actions.  
This first based on the notion of {\emph{sharpness}}
 (\cite{KasK16})
 measures
 how strongly a given action 
 satisfies
 the properness condition.  
The second takes a dynamical perspective 
 using volume estimates to assess how far the action deviates from being proper.  
This latter approach has emerged 
 as a key idea 
 in establishing temperedness criteria 
 for regular representations on $G$-spaces
 in recent work
 \cite{BeKoI, BeKoIII, BeKoII, BeKoIV}.  
Through this discussion, 
 we illustrate how geometric insights
 can inform analytic aspects of representation theory.

\vskip 1pc
\par\noindent
{\bf{Acknowledgement.}}
\newline
The author is grateful to the anonymous referee
 for a careful reading of the manuscript
 and providing valuable comments.  
This article is based 
 on the series of lectures
 delivered by the author
 as part of IMS Distinguished Visitor Lecture Series, 
 held during the Satellite Conference of the Virtual ICM 2022, 
 at the National University of Singapore,
{}from July 1 to 15, 2022.  
The author would like to express his gratitude
 to Hung Yean Loke, Angela Pasquale, 
 Tomasz Przebinda, Binyong Sun, 
 and Chengbo Zhu, 
 for their warm hospitality.

The author was partially supported by the JSPS under Grant-in Aid for Scientific Research (A) (JP23H00084), 
 Institut Henri Poincar{\'e} (Paris), 
 and by the Institut des Hautes {\'E}tudes Scientifiques (Bures-sur-Yvette).  

\section{Local to Global in Geometry}
\label{sec:global}
\subsection{Local to Global in Riemannian Geometry}
~~~\newline
We consider the interplay 
 between local and global geometric properties.  
\vskip 1pc
$\bullet$\enspace
{\emph{Local properties}} include curvature, 
 (T$_1$) topology, 
 locally homogeneous structure.  

$\bullet$\enspace
{\emph{Global properties}} include compactness, Hausdorffness, characteristic classes, 
diameter, 
 and the fundamental group.

\vskip 1pc
The study of how local geometric properties
 affect global structure
 has been one of the central themes
 of differential geometry over several decades,
 with particularly significant progress
 in the Riemannian setting.  
In contrast,
 relatively little is known 
 about global properties 
 in non-Riemannian geometry---arising, for example, from the space time model
 of relativity theory---or more generally in manifolds
 with indefinite metrics of arbitrary signature
 (see \cite{G25} and references therein).  
For instance, 
 the space form problem \cite[Sect.\ 2]{K01} is a long-standing problem
 in non-Riemannian geometry, 
 which includes the existence problem
 of a compact manifold $M$
 with constant sectional curvature
 for a given indefinite-metric of signature $(p,q)$, 
 see Conjecture~\ref{conj:G4} below.

We begin with a classical example of a local-to-global theorem
 in {\emph{Riemannian geometry}}.

\vskip 1pc
\begin{example}
[Bonnet--Myers]
\label{ex:Myers}
Let $(M, g)$ be an $n$-dimensional complete Riemannian manifold 
whose Ricci curvature satisfies $\operatorname{Ric}(g) \ge (n-1)C$
 for some positive constant $C$.  
Then $M$ is compact
 and its diameter is at most $\frac{\pi}{\sqrt{C}}$.  
\end{example}

This theorem tells us
 {\emph{global}} properties such as compactness 
 and the diameter are constrained
 by {\emph{local}} information---specifically
 such as the positivity of curvature---in Riemannian geometry.  

\vskip 2pc
What can be said about the local-to-global phenomena
 {\emph{beyond}} the traditional Riemannian setting?

\medskip 
\subsection{Preliminaries : Pseudo-Riemannian Manifolds}
~~~\newline

We review briefly some basic notions from pseudo-Riemannian geometry.  

\begin{definition}
\label{def:pseudo}
A {\emph{pseudo-Riemannian manifold}} $(M,g)$ is a smooth manifold
 equipped with a non-degenerate symmetric bilinear form
\[
g_x \colon T_x M \times T_x M \to {\mathbb{R}}
\quad
(x \in M)
\]
that depends smoothly on $x \in M$.  
\end{definition}

Let $(p,q)$ be the signature of $g_x$, 
 a non-degenerate symmetric bilinear form 
 on a $(p+q)$-dimensional manifold $M$.  
By Sylvester's law of inertia, 
 the signature is locally constant.  
We say that $(M,g)$ is
 a {\emph{Riemannian manifold}} if $q=0$, 
and is a {\emph{Lorentzian manifold}} if $q=1$.

Just as in the Riemannian case, 
 pseudo-Riemannian manifolds $(M, g)$ also admit natural definitions
 of the Levi-Civita connection, geodesics, and curvature.  
For example, 
the curvature tensor $R$ and the sectional curvature $\kappa$
 for $X, Y \in T_x M$ are given by 
\begin{align*}
  R(X,Y):=&[\Delta_X, \Delta_Y]-\Delta_{[X,Y]}, 
\\
\kappa(X,Y):=&\frac{g_x(R(X,Y)Y,X)}{g_x(X,X) g_x(Y,Y)-g_x(X,Y)^2}.  
\end{align*}

\begin{example}
\label{ex:Rpq_surface}
(1)\enspace(Flat case)\enspace
We equip ${\mathbb{R}}^{p+q}$ 
 with the pseudo-Riemannian structure
\[
   dx_1^2 + \cdots +  dx_p^2 - dx_{p+1}^2 - \cdots -dx_{p+q}^2
\]
which has the signature $(p,q)$, 
 and denote the resulting space by ${\mathbb{R}}^{p, q}$.

It is a flat space;
 that is, 
 the curvature tensor satisfies $R \equiv 0$.  
In the case where $q=1$,
 ${\mathbb{R}}^{p,q}$ is a Lorentzian manifold 
 known as the {\emph{Minkowski space}}.  
\newline
(2)\enspace(Pseudo-Riemannian space forms)\enspace
The flat pseudo-Riemannian structure on ${\mathbb{R}}^{p,q}$ remains 
 non-degenerate when restricted to the hypersurfaces
\begin{align*}
X(p-1,q)_+=&\{x \in {\mathbb{R}}^{p+q}: x_1^2 +\cdots + x_p^2-x_{p+1}^2 - \cdots - x_{p+q}^2=1\}, 
\\
X(p,q-1)_-=
&\{x \in {\mathbb{R}}^{p+q}: x_1^2 +\cdots + x_p^2-x_{p+1}^2 - \cdots - x_{p+q}^2=-1\}.  
\end{align*}
These give rise to pseudo-Riemannian manifolds
 of signature $(p-1,q)$
 with constant sectional curvature $\kappa \equiv 1$, 
 and of signature $(p,q-1)$
 with $\kappa \equiv-1$, respectively.  
\newline
(3)\enspace(Hyperbolic space: ${\mathbb{H}}^n=X(n,0)_-$)\enspace
The hypersurface
\[
{\mathbb{H}}^n
:=\{
(x_1, \cdots, x_{n+1})
:
x_1^2+\cdots+x_n^2-x_{n+1}^2=-1
\}
\]
inherits a Riemannian structure from the ambient Minkowski space ${\mathbb{R}}^{n,1}$, 
 and has a constant sectional curvature $\kappa \equiv -1$.  
\newline
(4)\enspace
(De Sitter space: $\operatorname{d S}^n=X(n-1,1)_+$)\enspace
The hypersurface of ${\mathbb{R}}^{n,1}$, 
\[
\operatorname{dS}^n:=\{
(x_1, \cdots, x_{n+1})
:
x_1^2+\cdots+x_n^2-x_{n+1}^2=1
\}
\]
 inherits a Lorentzian metric from the ambient Minkowski space ${\mathbb{R}}^{n,1}$, 
 and has a constant sectional curvature $\kappa \equiv 1$.  
More generally, 
 a complete Lorentzian manifold
 of constant positive sectional curvature 
 is called 
 a {\emph{de Sitter manifold}}, 
 or a {\emph{relativistic spherical space}}, 
 as it serves as a Lorentzian analog
 of sphere geometry.  
\newline
(5)\enspace
(Anti-de Sitter space: $\operatorname{AdS}^n=X(n-1,1)_-$)\enspace
This is a special case of the preceding example.  
The hypersurface
\[
\operatorname{A d S}^n:=\{
(x_1, \cdots, x_{n+1})
:
x_1^2+\cdots+x_{n-1}^2-x_n^2-x_{n+1}^2=-1
\}
\] 
 inherits a Lorentzian metric from ${\mathbb{R}}^{n-1,2}$
 and has constant sectional curvature 
 $\kappa \equiv -1$.  
It is regarded as a Lorentzian analog of the hyperbolic space.  
More generally, 
 a complete Lorentzian manifold
 with constant sectional curvature $\kappa \equiv -1$
 is called an {\emph{anti-de Sitter manifold}}.  
\end{example}

\begin{remark}
\label{rem:Rqp}
In Example~\ref{ex:Rpq_surface} (2), 
 changing the signature of the flat pseudo-Riemannian structure
 of the ambient space ${\mathbb{R}}^{p+q}$ causes
 the signatures
 of the induced pseudo-Riemannian metrics
 on the hypersurfaces
 $X(p-1,q)_+$ and $X(p,q-1)_-$
 to change from $(p-1,q)$ to $(q, p-1)$, 
 and from $(p,q-1)$ to $(q-1,p)$, respectively.  
Furthermore, 
the sectional curvature is reversed in sign.  
\end{remark}


\subsection{The Calabi--Markus Phenomenon}
~~~\newline
In contrast to Riemannian geometry, 
 as illustrated by the Bonnet--Myers theorem
 (Example~\ref{ex:Myers})
 the global geometry 
 of pseudo-Riemannian manifolds exhibits markedly
 different behavior:

\begin{theorem}
[Calabi--Markus \cite{CM}]
\label{thm:CM}
Every relativistic spherical space ({\it{i.e.,}} a de Sitter manifold) is non-compact.  
Furthermore, 
 if the dimension is greater than two, 
 its fundamental group is finite.  
\end{theorem}

\section{Basic Problems on Discontinuous Groups for $G/H$}
\label{sec:discgp}

When the {\emph{homogeneous structure}}
 is regarded as a local property, 
 {\emph{discontinuous groups}}
 (Definition~\ref{def:G1})
 govern the global geometry.  
The study of discontinuous groups
 beyond the Riemannian setting
 is a relatively young and rapidly evolving field
 in group theory
 interacting with topology, 
 differential geometry, 
 representation theory, 
 ergodic theory, 
 and number theory, as well as other areas of mathematics.  
An early exposition of this subject
 can be found in the lecture notes \cite{K97}, 
 and a more recent account is provided, 
 for instance, in \cite{G25}.

This theme was also highlighted as a new direction 
 for future research
 looking ahead to the 21st century
 on the occasion 
 of the World Mathematical Year 2000
 by Margulis \cite{Mr}
 and the author \cite{K01}
 with both works including collection of open problems.  
Over the past thirty years, 
 there have been remarkable developments
 employing a variety of methods.  
Nevertheless, 
 several fundamental problems remain unsolved
 (\cite{KConj23}).

In this section, 
 we lay the groundwork for these problems, 
 which will be formulated
 more explicitly 
 in Sections~\ref{sec:proper}---{\ref{sec:cocompact},
 by illustrating the basic ideas
 through simple examples.

\subsection{Discontinuous Groups for Acting on Manifolds $X$}
~~~\par
Beyond the Riemannian context, 
 it is crucial to clearly distinguish
 between {\emph{discrete subgroups}} and {\emph{discontinuous groups}}.

In many cases, 
 a discontinuous group $\Gamma$ is realized
 as a subgroup of $G$ acting on a manifold.  
Accordingly, 
 we shall define discontinuous groups
 within this framework.  
Nevertheless, in contexts
 where $G$ plays no essential role, 
 we may omit the ambient group $G$
 and simply take $G=\Gamma$.  

\begin{definition}
[Discontinuous Group]
\label{def:G1}
Let $G$ be a Lie group acting
 on a manifold $X$.  
A discrete subgroup $\Gamma$ of $G$ 
 is called a {\emph{discontinuous group}} for $X$
 if $\Gamma$ acts properly discontinuously and freely on $X$.  
See Definition \ref{def:ppf} below.  
\end{definition}

The quotient space $X_{\Gamma}:=\Gamma \backslash X$, 
 by a discontinuous group $\Gamma$, 
 is a (Hausdorff) manifold.  
Moreover, 
 any $G$-invariant local geometric structure on $X$
 descends to $X_{\Gamma}$
 via the covering map $X \to X_{\Gamma}$
 (see Proposition~\ref{prop:Deck}).

Such quotients $X_{\Gamma}$ are examples of complete $(G,X)$-{\it{manifolds}}
 in the sense of Ehresmann and Thurston.


\subsection{Basic Notion $\cdots$ Proper Action}
~~~
\newline
We extend Theorem~\ref{thm:CM}
 to a broader setting
 formulated in the language of groups.  
To this end, 
 we briefly review some basic notions
 in the theory of transformation groups.

Let $L$ be a locally compact group, 
 and $X$ a locally compact topological space.  
Suppose that $L$ acts {\emph{continuously}} on $X$, 
 {\it{i.e.}}, the action map
\[
   L \times X \to X, \quad
  (g,x) \mapsto g x
\]
 is continuous.

For a subset $S \subset X$, 
 we define a subset $L_S \subset L$ by 

\begin{equation*}
L_S:= L_{S \to S}=\{\gamma \in L: \boldsymbol\gamma S \cap S \ne \emptyset\}.
\end{equation*}

If $S$ is a singleton $\{x\}$, 
 then $L_{\{x\}}$ coincides with the stabilizer group $L_x$ of the point $x \in X$.  
In general, 
 $L_S$ is merely a subset of $L$.

\begin{figure}[H]
\centering
\raisebox{30pt}{\includegraphics[width=30mm]{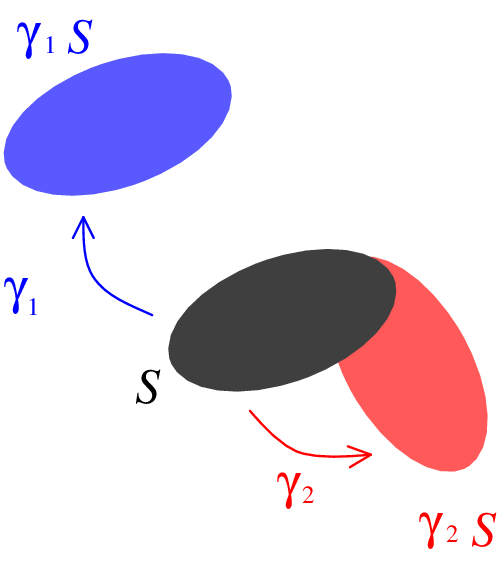}}
\caption{$\BLUE{{\boldsymbol\gamma}_1} \not \in L_S \ni \RED{{\boldsymbol\gamma}_2}$}
\end{figure}

Continuous actions possessing with the properties:
\[\text{
$L_S$ is \lq\lq{small}\rq\rq\
 whenever $S$ is \lq\lq{small}\rq\rq
}
\]
are precisely formulated 
 and given the following names.

\begin{definition}
\label{def:ppf}
An action of $L$ on $X$ is called
\begin{alignat*}{3}
&\text{{\emph{free}}}
&&\text{if $L_S$ is a singleton}\quad
&&\text{for any singleton $S$;}
\\
&\text{{\emph{properly discontinuous}}}
\quad
&&\text{if $L_S$ is finite}
&&\text{for any compact subset $S$;}
\\
&\text{{\emph{proper}}}
&&\text{if $L_S$ is compact}\quad
&&\text{for any compact subset $S$.} 
\end{alignat*}
\end{definition}


\subsection{Proper Maps and Proper Actions}
~~~\newline
Let $X$ and $Y$ be Hausdorff, locally compact spaces.  
In this section, 
 we take a closer look at some basic properties 
 of proper actions.  

\begin{definition}
\label{def:propermap}
A continuous map $f \colon X \to Y$ 
 is called {\emph{proper}}
 if the preimage $f^{-1}(S)$ of any compact subset $S \subset Y$
 is compact.  
\end{definition}

It is worth noting that any proper map is a closed map
 (see {\it{e.g.}}, \cite[Chap.\ I, Sect.\ 10, Prop.\ 1]{B98}).  
Indeed, 
 let $C$ be a closed subset of $X$, 
 and let $y \not \in f(C)$.  
Choose an open neighbourhood $V$ of $y$
 such that its closure $\overline V$ is compact.  
Then the set $E:=C \cap f^{-1}(\overline V)$ is compact, 
 and hence $f(E)$ is closed.  
It follows that the set $U:=V \setminus f(E)$ is an open neighbourhood of $y$, 
 disjoint from $f(C)$.  
Thus, 
 $f(C)$ is closed.

For subsets $S$ and $T$ of $X$, 
we define
\[
 L_{S \to T}:=\{\gamma \in L: \gamma S \cap T \ne \emptyset \}.  
\]
The proof of the following lemma is straightforward
 and is therefore omitted.  

\begin{definition-lemma}
[Proper Action]
\label{def:properaction}
Let $X$  be a locally compact topological space, 
 on which a locally compact group $G$ acts continuously.  
Then the following four conditions are equivalent:

{\rm{(i)}}\enspace
The action of $L$ on $X$ is proper
  in the sense of Definition \ref{def:propermap}.  

{\rm{(ii)}}\enspace
The map $\varphi \colon L \times X \to X \times X$
 defined by $(g,x) \mapsto (x,gx)$
is proper.  

{\rm{(iii)}}\enspace
For any compact $S,T \subset X$, 
 the set $L_{S \to T}$ is compact.  

{\rm{(iv)}}\enspace
For any compact subset $S \subset X$, 
 the set 
$L_S$ ($\equiv L_{S \to S}$) is compact.  
\end{definition-lemma}

See also Lemma~\ref{lem:doi_250427}
 for alternative characterization of proper actions from the
 perspective of measure theory.


\subsection{Proper $+$ Discrete $=$ Properly Discontinuous}
~~~\newline
When the group $L$ is discrete, 
 the action of $L$ is proper
 if and only if it is properly discontinuous, 
 since a discrete set is compact
 if and only if it is finite.

Furthermore, 
 the stabilizer $L_x$ is finite for every $x \in X$
 in this case.  
Thus, 
 among the three properties listed in Definition~\ref{def:ppf}, 
 understanding of proper actions
 in greater depth is of particular importance.


\subsection{Discontinuous Group and Covering Transformation Group}
Suppose that $X$ is a locally compact, 
Hausdorff space, 
 on which a discrete group $\Gamma$ acts continuously.  

\begin{definition}
[Discontinuous Group for $X$]
\label{def:discont}
A discrete group $\Gamma$ is called a {\emph{discontinuous group for $X$}}
 if the action of $\Gamma$
 on $X$ is properly discontinuous and free.  
\end{definition}

Let $\Gamma \backslash X$ denote the quotient space, 
 {\it{i.e.}}, 
 the set of $\Gamma$-orbits in $X$, 
 equipped with the quotient topology induced
 by the natural projection 
 $q_{\Gamma} \colon X \to \Gamma \backslash X$.  
The following result is a classical fact from general topology
 (see {\it{e.g.}}, \cite[Chap.\ 3, Sect.\ 3.5]{T97}).  

\begin{proposition}
\label{prop:Deck}
If $\Gamma$ is a discontinuous group for a topological manifold, 
 then the quotient space $\Gamma \backslash X$ carries a manifold structure
such that the quotient map $q_{\Gamma} \colon X \to \Gamma \backslash X$ becomes a regular covering.  
Moreover, 
 any $\Gamma$-invariant local geometric structure on $X$ 
 descends to $\Gamma \backslash X$
 via $q_{\Gamma}$.  
\end{proposition}

\begin{remark}
The key condition in Definition~\ref{def:discont}
 is that the action is properly discontinuous; 
 freeness is of secondary importance.

There are two main reasons for this.  
First, 
 suppose that $\Gamma$ acts properly discontinuously on $X$.  
Then the singularities of the quotient space $X_{\Gamma}$ are \lq\lq{mild}\rq\rq, 
 in the sense that $X_{\Gamma}$ is locally a finite group quotient
 of Euclidean space, 
 called {\emph{V-manifold}}
 in the sense of Satake \cite{S56}
 or an {\emph{orbifold}} in the sense of Thurston.

Second, 
 if $\Gamma$ is a finitely generated linear group, 
 then there exists a finite-index subgroup $\Gamma'$ 
 $\subset \Gamma$
 such that $\Gamma'$ is torsion-free
 by a theorem of Selberg.  
In particular,
 the $\Gamma'$-action is free
 and properly discontinuous, 
 provided that the $\Gamma$-action is properly discontinuous.

In \cite[Def.\ 2.5]{K97}, 
 we did not require freeness
 in the definition of discontinuous groups, 
 thereby allowing $X_{\Gamma}=\Gamma \backslash X$ to be an orbifold.  
\end{remark}

We provide some typical examples of Proposition~\ref{prop:Deck}.  

\vskip 2pc
\begin{example}
\label{ex:univ_cover_M}
Suppose that $M$ is a pseudo-Riemannian manifold.  
Let $X$ be its universal covering 
equipped with the pull-back pseudo-Riemannian structure
 via the covering map $p \colon X \to M$.  
Let $G=\operatorname{Isom}(X)$ denote the isometry group of $X$, 
 and let $\Gamma$ be the fundamental group of $M$, 
 based on a point $o=p(\widetilde o) \in M$.  
Then $G$ admits the structure of a Lie group
 acting smoothly on $X$.  
Furthermore,
 $\Gamma$ is regarded as a subgroup of the Lie group $G$, 
 and is a discontinuous group for $X$
 with natural isomorphism $X_{\Gamma} \simeq M$.  
\end{example}

As a classical example illustrating Example~\ref{ex:univ_cover_M}, 
 we recall the uniformization
 of a compact Riemann surface $\Sigma_g$.

\begin{example}
[Uniformization Theorem of Klein--Poincar{\'e}--Koebe]
\label{ex:Sg}
~~~\newline
Let $\Sigma_g$ be a compact Riemann surface 
 of genus $g \ge 2$, 
 and let $\Gamma$ denote its fundamental group $\pi_1(\Sigma_g)$, 
 often referred to as a {\emph{surface group}}.  
Then the universal covering space of $\Sigma_g$ is biholomorphic
 to the Poincar{\'e} upper half plane
\[
   {\mathbb{H}}=\{z \in {\mathbb{C}}: \operatorname{Im}z >0\}.  
\]
The group $PSL(2,{\mathbb{R}}) =SL(2,{\mathbb{R}})/\{\pm I_2\}$ acts 
holomorphically
 and transitively on ${\mathbb{H}}$ 
 via linear fractional transformations
 $z \mapsto (cz+d)^{-1}(az+b)$.  
There is a natural diffeomorphism
\[
  {\mathbb{H}} \simeq PSL(2,{\mathbb{R}})/PSO(2)=:G/K.  
\]
The quotient $\Gamma \backslash {\mathbb{H}}\simeq \Gamma \backslash G/K$
 can be naturally identified with the original surface $\Sigma_g$.  
\end{example}

In this example, 
 $G/K=PSL(2,{\mathbb{R}})/PSO(2)$ is a Riemannian symmetric space.  
More generally, 
 the following result, 
 which goes back to {\'E}.\ Cartan, 
 provides a bridge 
 between the geometric and group-theoretic definition 
 of symmetric spaces.

\begin{proposition}
[Affine Locally Symmetric Space]
Any complete affine locally symmetric space
 is of the form $\Gamma\backslash G/H$, 
 where $G$ is a Lie group,
 $H$ is an open subgroup
 of the fixed point subgroup
 of an involution of $G$, 
 and $\Gamma$ is a discrete subgroup 
 acting properly discontinuously and freely on the symmetric space $G/H$.  
\end{proposition}


\subsection{Isometric Actions: Riemannian Geometry}
~~~\newline
Let $\operatorname{Isom}(X)$ denote 
 the group of isometries of a Riemannian manifold, 
 or more generally, 
 of a pseudo-Riemannian manifold $X$.  
Then $\operatorname{Isom}(X)$ is a Lie group.

In Example~\ref{ex:Sg}, 
 $G:=\operatorname{Isom}(\Sigma_g) \simeq PSL(2,{\mathbb{R}})$, 
 and $\Gamma$ can be regarded
 as a {\emph{discrete}} subgroup of $G$.

In this subsection, 
 we prove a converse statement, 
 that is, 
 any discrete subgroup $\Gamma$ 
 of $\operatorname{Isom}(X)$ acts properly discontinuously
 on $X$
 if $X$ is a {\emph{Riemannian manifold}}.

For two topological spaces $X$ and $Y$, 
 let $C(X,Y)$ denote the set 
 of all continuous maps from $X$ to $Y$.  
We recall that the {\emph{compact-open topology}} 
 on the set $C(X,Y)$ 
 is a topology defined
 by the subbase
\[
  W(S, V):=\{f \in C(X,Y): f(S) \subset V\}, 
\]
 where $S$ is a compact subset of $X$
 and $V$ is an open subset of $Y$.

The compact-open topology on $C(X,Y)$ is Hausdorff 
 if $Y$ is Hausdorff.

\begin{proposition}
[Isometric Transformations in Metric Spaces]
\label{prop:isomR}
~~~\newline
Suppose that $X$ is a locally compact, 
 separable, 
 complete metric space
 such that $X$ has a Heine--Borel property, 
 that is, 
 every bounded closed set is compact.  
Let $\Gamma$ be a group of isometries of $X$ 
 endowed with compact-open topology.  
Then the following two conditions on $\Gamma$ are equivalent:
\newline
{\rm{(i)}}\enspace
$\Gamma$ is a discrete group.  
\newline
{\rm{(ii)}}\enspace
$\Gamma$ acts properly discontinuously on $X$.    
\end{proposition}

\begin{proof}
We first prove the easier direction (ii) $\Rightarrow$ (i).  
It suffices to show
 that, 
 for any $\gamma \in \Gamma$, 
 there exists an open set $\Gamma \subset W$
 such that $\sharp (\Gamma \cap W)< \infty$, 
 assuming (ii).  
Take any $x \in X$ and any open neighbourhood $V$ of $\gamma \cdot x$
 such that the closure $\overline V$ is compact.  
Let $W:=W(\{x\}, V)$.  
Then $\gamma \in W=\Gamma_{\{x\} \to V} \subset \Gamma_{\{x\} \to \overline V}$.  
On the other hand, 
 $\Gamma_{\{x\} \to \overline V}$ is finite 
 because the $\Gamma$-action on $X$ is properly discontinuous.  
Hence (ii) $\Rightarrow$ (i) is shown.

(i) $\Rightarrow$ (ii):\enspace
This is a non-trivial part.  
The argument uses a variation of the Ascoli--Arzela theorem to the metric space $(X,g)$.  
For completeness, 
 we include a full proof below.

Suppose, 
 on the contrary, 
 that the action of a group $\Gamma$ of isometries
 is not properly discontinuous.  
Then there exists a compact subset $S \subset X$, 
 an infinite sequence $\{\gamma_k\} \subset \Gamma$, 
 a sequence $\{s_k\} \subset S$
 such that $\gamma_k \cdot s_k \in S$
 for all $k \in {\mathbb{N}}$.  
We shall show
 that $\{\gamma_k\}$ cannot be discrete in the compact-open topology of $\operatorname{Isom}(X, g)$.

For $x \in S$, 
 we set $M(x)\equiv M(x;S):=\underset{a \in S}\max\, d (x,a)$.  
For any $x \in X$, 
 one has 
\[
  d(x, \gamma_k \cdot x) \le d(x, \gamma_k \cdot s) + d(\gamma_k \cdot s, \gamma_k \cdot x) \le 2 M(x).  
\]
Since every bounded closed set is compact, 
 $\{\gamma_k \cdot x\}$ has a convergent subsequence in $X$.

We take a countable and dense subset $\{x_j\}_{j \in {\mathbb{N}}}$ in $X$.  
By Cantor's diagonal argument, 
 there exist a subsequence of positive integers
 $k_1< k_2< \cdots$
 such that $\gamma_{k_{\ell}}x_j$ converges as $\ell$ tends to infinity
 for every $j \in {\mathbb{N}}$.  
For simplicity, 
 we continue to denote the subsequence $\gamma_{k_{\ell}}$
 by $\gamma_k$.

We claim that the sequence of maps $\gamma_k|_C$ converges uniformly
 on any compact subset $C$ in $X$.  
To see this, 
 let $\varepsilon >0$.  
Since $C$ is compact, 
 one can take $N>0$
such that for any $x \in C$
 there exists $j\equiv j(x)\in \{1,2,\dots, N\}$
 with $d(x,x_j)<\frac\varepsilon 3$.  
We take $T>0$ 
 such that 
\[
   d(\gamma_k \cdot x_i, \gamma_{k'} \cdot x_i)<\frac\varepsilon 3
\]
 for any $k,k'  \ge T$
 and for any $1 \le i \le N$.  
Then for any $x \in C$
\begin{align*} 
   d(\gamma_k \cdot x, \gamma_{k'} \cdot x)
 \le & d(\gamma_k \cdot x, \gamma_k \cdot x_j) + d(\gamma_k \cdot x_j, \gamma_{k'} \cdot x_j) + d(\gamma_{k'} \cdot x_j, \gamma_{k'} \cdot x)
\\
<&\frac \varepsilon 3 + \frac \varepsilon 3 + \frac \varepsilon 3
=\varepsilon, 
\end{align*}
because $\gamma_k$ is an isometry.  
Hence $\gamma_k \cdot x$ converges to an element, 
say $\gamma_C \cdot x$ in $X$.  
By taking a sequence $C_1 \subset C_2 \subset \cdots$
 of compact subsets in $X$
 with $X=\underset{i=1}{\overset{\infty}\cup} C_i$, 
 one sees that the map $\gamma_{C_j}|_{C_i}$ coincides with $\gamma_{C_i}$
 whenever $C_i \subset C_j$
 for $i \le j$, 
 because of the uniqueness of the limit.  
Hence $\gamma \colon X \to X$
 is defined as the inductive limit of $\gamma_{C_i}$.

We claim 
 that the limiting map $\gamma$ lies in $\operatorname{Isom}(X,g)$.  
In fact, 
 for any $x, x' \in X$, 
 one has
\begin{align*}
d(\gamma \cdot x, \gamma \cdot x')
=&d(\underset{k \to \infty} \lim \gamma_k \cdot x, 
    \underset{k \to \infty} \lim \gamma_k \cdot x')
\\
=&\underset{k \to \infty} \lim d(\gamma_k \cdot x, \gamma_k \cdot x')
=d(x,x').  
\end{align*}
Hence $\gamma$ is an isometry.  
Moreover, 
 the sequence $\{\gamma_k^{-1}\}$ yields $\gamma^{-1}$
 as its limit, 
 showing that the isometry $\gamma \colon X \to X$ is a surjective map.  
Since $\gamma_k$ converges to $\gamma$
 with respect to the compact-open topology, 
 $\Gamma$ is not closed
 in $\operatorname{Isom}(X, g)$.  
Since a discrete group is closed
 (see {\it{e.g.,}} \cite[(5.10)]{HR63}), 
 the reverse implication (i) $\Rightarrow$ (ii) is proved.  
\end{proof}

\subsection{Isometric Actions in Pseudo-Riemannian Geometry}
~~~\newline
The group of isometries of any pseudo-Riemannian manifold is a Lie group.  
However, 
 the proof of Proposition~\ref{prop:isomR} relies heavily
 on the {\emph{positive-definiteness}}
 of the metric on $X$.  
This leads to the following question:
\begin{question}
[Action of Isometric Discrete Group]
\label{q:discont_discrete}
Does the equivalence (i) $\Leftrightarrow$ (ii)
Proposition~\ref{prop:isomR} still hold in the pseudo-Riemannian setting?
\end{question}

Unfortunately, 
 an analogue of Proposition~\ref{prop:isomR}
 fails in pseudo-Riemannian geometry.

Let $X$ be a pseudo-Riemannian manifold.  
Let $\operatorname{Isom}(X)$ denote 
 the group of isometries, 
 and let $\Gamma$ be a subgroup of $\operatorname{Isom}(X)$.  
Then the implication (ii) $\Rightarrow$ (i)
 in Proposition~\ref{prop:isomR} remains true, 
 but the converse implication (i) $\Rightarrow$ (ii) does not, 
 as demonstrated in the following example.

\begin{example}
[Isometric but Non-Proper Action]
\label{ex:Lorentz_isometry}
Let $\Gamma := {\mathbb{Z}}$ act on $X:={\mathbb{R}}^2$ by
\[
  (x,y) \mapsto (e^n x, e^{-n} y)
  \qquad
\text{for $n \in {\mathbb{Z}}$.}
\]
We first observe
 that there does not exist a metric $d$ on $X={\mathbb{R}}^2$
 with respect to which $L$ acts isometrically.  
In fact, 
 suppose such metric $d$ existed.  
Let $o:=(0,0)$ and $p:=(0,1)$.  
Then for $t \in {\mathbb{Z}}$, 
 we compute 
\[
   d(o,p)=d (t \cdot o, t \cdot p)=d(o, (o, e^{-t})).  
\]
Taking the limit as $t \to \infty$, 
 we find $d(o,p)\to 0$, 
 hence $d(o,p)=0$, 
 which contradicts the positive definiteness of $d$.

While no $\Gamma$-invariant {\emph{Riemannian}} metric exists on $X$, 
 there does exist a $\Gamma$-invariant {\emph{Lorentzian structure}} on $X$.  
Indeed, 
 consider the two-dimensional Minkowski space ${\mathbb{R}}^{1,1}$
 with coordinates 
\[
(x,y):=(x_1+x_2, x_1-x_2), 
\]
 where the Lorentzian metric tensor is given
 by $d x d y= dx_1^2-dx_2^2$.  
Then the $\Gamma$-action preserves the Lorentzian structure.  
Thus, 
 $\Gamma$ forms a discrete group of isometries
 of a Lorentzian manifold, 
 but the action is not properly discontinuous, 
 since the origin $o$ is fixed 
 by all elements of $\Gamma$.  
\end{example}

\medskip 
This example will be revisited from different perspectives
 throughout the paper.  
For instance, 
 it will appear in Example \ref{ex:R2b} from a group-theoretic point of view, 
 and again, 
 in Example \ref{ex:SL2qX} in the context of dynamical volume estimates.

\subsection{A Large Isometry Group}
~~~\newline
As mentioned earlier, 
 the isometry group of any pseudo-Riemannian manifold 
 is a Lie group.  
Here, 
 we present a representative class of pseudo-Riemannian manifolds
 whose isometry group act transitively.  

\begin{proposition}
\label{prop:redGH}
Let $G \supset H$ be a pair of real reductive Lie groups, 
 and let $X:=G/H$.  
Then the homogeneous space $X$ admits a pseudo-Riemannian structure
with respect to which $G$ acts isometrically.  
\end{proposition}

\begin{proof}
By a theorem of Mostow, 
 there exists a Cartan involution $\theta$ of $G$
 such that $\theta H=H$.  
Let ${\mathfrak{g}}={\mathfrak{k}}+{\mathfrak{p}}$
 be the corresponding Cartan decomposition 
 of the Lie algebra.  
Take an $\operatorname{Ad}(G)$-invariant, non-degenerate
 symmetric bilinear form $B$
 on ${\mathfrak{g}}$
 such that $B|_{{\mathfrak{k}}\times {\mathfrak{k}}}$
 is negative definite, 
 $B|_{{\mathfrak{p}}\times {\mathfrak{p}}}$ is positive definite, 
 and $B|_{{\mathfrak{k}}\times {\mathfrak{p}}} \equiv 0$
 ({\it{e.g.,}} the Killing form if ${\mathfrak{g}}$ is semisimple).

Then, 
 $B$ induces an $H$-invariant, 
 non-degenerate symmetric bilinear form $\overline B$
 on the quotient space
\[
   {\mathfrak{g}}/{\mathfrak{h}}
   =
   {\mathfrak{k}}/({\mathfrak{h}} \cap {\mathfrak{k}})
   \oplus
   {\mathfrak{p}}/({\mathfrak{h}} \cap {\mathfrak{p}}), 
\]
of signature $(d(X), e(X))$, 
 where 
\begin{equation}
\label{eqn:dx}
\text{$d(X):= \dim {\mathfrak{p}}/({\mathfrak{h}} \cap {\mathfrak{p}})$
 \quad and \quad $e(X):=\dim {\mathfrak{k}}/({\mathfrak{h}} \cap {\mathfrak{k}})$.}
\end{equation}
Identifying ${\mathfrak{g}}/{\mathfrak{h}}$ 
 with the tangent space $T_o X$
 at $o:= e H \in X$, 
 we extend this bilinear form $\overline B$
 to each $T_{g \cdot o}X$ for $g \in G$ 
 via the left translation map
 $d L_g \colon T_o X \to T_{g \cdot o}X$.  
This extension is well-defined
 because the bilinear form $\overline B$ is $H$-invariant.

Consequently, 
 $X$ carries a pseudo-Riemannian structure
 of signature $(d(X),e(X))$, 
 on which $G$ acts isometrically by construction.  
\end{proof}

The numbers $d(X)$ and $e(X)$ also have natural geometric interpretations:
 the homogeneous space
 $X=G/H$ admits a $K$-equivariant smooth vector bundle structure
\[
  {\mathbb{R}}^{d(X)} \to X \to Y, 
\]
 where the base space $Y$ is the compact manifold
 $K/H \cap K$ of dimension $e(X)$, 
 see \cite{K89} for example.

Here are some classical examples:

\begin{example}
[Riemannian Symemtric Space]
Let $H=K$, 
 a maximal compact subgroup of $G$.  
Then $d (X)=\dim {\mathfrak{p}}$
 and $e(X)=0$.  
Hence the pseudo-Riemannian structure on $X=G/K$
 is positive definite.  
The resulting Riemannian manifold $G/K$ is called
 a {\emph{Riemannian symmetric space}}.  
\end{example}

\begin{example}
[Pseudo-Riemannian Space Form]
\label{ex:Xpq_group}
Let $(G, H)=(O(p,q), O(p-1,q))$, 
 and $X=G/H$.  
By a straightforward computation, 
 we have  $d(X)=q$, 
 $e(X)=p-1$.  
Thus, the pseudo-Riemannian manifold $X$
 is of signature $(q, p-1)$, 
 and can be identified with the hypersurface
\[
  X(p-1,q)_+=\{x \in {\mathbb{R}}^{p+q}:x_{1}^2 + \cdots +  x_{p}^2 -  x_{p+1}^2 - \cdots -  x_{p+q}^2=1\} 
\]
in ${\mathbb{R}}^{p,q}$.  
The manifold $X(p-1, q)_+$ is diffeomorphic
 to a vector bundle 
 over the sphere $S^{p-1}$
 with fiber ${\mathbb{R}}^q$.  
Note that the signature $(d(X), e(X))$ is opposite
 to the convention 
 used in Example~\ref{ex:Rpq_surface}.  
This sign reversal is explained in Remark~\ref{rem:Rqp}.  
\end{example}

The de Sitter space $\operatorname{dS}^n=X(n-1,1)_+$ is a special case
 of Example~\ref{ex:Xpq_group}.  
The Calabi--Markus theorem 
 (Theorem~\ref{thm:CM}) can be reformulated
 in group-theoretic terms as follows:

\begin{theorembis}{thm:CM}
[Calabi--Markus {\cite{CM}}]
Let $(G,H)=(O(n,1),O(n-1,1))$.  
If a discrete subgroup $\Gamma \subset G$ acts 
 properly discontinuously on $G/H$, 
 then $\Gamma$ must be finite.  
\end{theorembis}

\subsection{Elementary Consequences of Proper Actions}
~~~\newline
We begin by discussing some elementary consequences 
 of proper actions 
 in the general setting where a
 locally compact group
 acts continuously on a locally compact Hausdorff space. 

\begin{proposition}
\label{prop:118}
Suppose that a locally compact group $L$ acts properly
 on a locally compact Hausdorff space.  
Then the following hold:

{\rm{(1)}}\enspace
The quotient space $L\backslash X$ is Hausdorff in the quotient topology;

{\rm{(2)}}\enspace 
Each orbit $L \cdot x$ is closed in $X$ for all $x \in X$;

{\rm{(3)}}\enspace
Each isotropy subgroup $L_x$ is compact for all $x \in X$.  
\end{proposition}

The condition (2) is equivalent to
 the statement that the quotient space $L \backslash X$ satisfies
 the (T$_1$) separation axiom.  
Thus, the implication (1) $\Rightarrow$ (2) in Proposition~\ref{prop:118} is
 immediate.  
We note that the Hausdorff property
 is global in nature, 
 whereas the (T$_1$) property is local.

\begin{definition}
\label{def:CI}
A continuous action is said to have the (CI) {\emph{property}}
 if the condition (3) in Proposition~\ref{prop:118} is satisfied.  
\end{definition}

The (CI) property is an abbreviation
 introduced by the author \cite{KICM90sate}, 
 standing for \lq\lq{Compact Isotropy}\rq\rq,
 which refers to the condition
 that all isotropy subgroup are compact.

Let $\varphi \colon L \times X \to X \times X$, $(g,x) \mapsto (x, g x)$
 be the action map, 
 as used in Definition-Lemma~\ref{def:properaction} (ii).  
If the $L$-action on $X$ is proper, 
 then $\varphi$ is a closed map.  

\begin{proof}
[Proof of Proposition~\ref{prop:118}]
(1)\enspace
Let $\overline X:= L \backslash X$ denote the quotient space, 
 and let $\pi \colon X \to \overline X$
 be the quotient map.  
To show that $\overline X$ is Hausdorff, 
 it suffices to prove that the complement
\[
   \overline X \times \overline X \setminus \operatorname{diag} (\overline X)
\]
 is open.  
Equivalently, 
 it suffices to show
 that the preimage of the diagonal under $\pi \times \pi$, 
 {\it{i.e.}}, 
\[
(\pi \times \pi)^{-1}(\operatorname{diag} (\overline X))=\operatorname{Image} \varphi
\]
is closed in $X \times X$.  
Since the action is proper, 
 $\varphi$ is a proper map
 and hence closed, 
 which implies that $\operatorname{Image}(\varphi)$ is closed.

(2)\enspace
Again, 
 since $\varphi$ is a closed map, 
 $\varphi(L \times \{x\}) = \{x\} \times L \cdot x$ is closed.  

(3)\enspace
Since $\varphi$ is proper, 
 $L_x=L_{\{x\} \to \{x\}}$ is compact.  
\end{proof}


\subsection{Subtle Examples (Hausdorff $\ne (\text{T}_1)$)}
~~~\newline
One may naturally ask
 whether the converse of the statements
 of Proposition~\ref{prop:118} also holds.  
In particular, 
 we consider 
 whether the following implications are generally valid:

\begin{tabular}{lll}
(A)&
free action 
&{ $\overset{?}{\Longrightarrow}$}
{proper action, }
\rule{0pt}{15pt}
\\
(B)&
any orbit is closed 
&{ $\overset{?}{\Longrightarrow}$}
$L \backslash X$ is Hausdorff.  
\\ 
\end{tabular}

However, 
 neither of these statements hold
 in the setting 
 where $X$ is a locally compact topological space
 endowed with a continuous action 
 of a locally compact group $L$.

\begin{example}
\label{ex:R2a}
Let $L:={\mathbb{R}}$, the additive group,  act on 
\[
X:={\mathbb{R}}^2 \setminus \{(0, 0)\}
\quad
\text{by }
(x,y)
\mapsto
(e^t x, e^{-t} y)
\quad
\text{for $t \in {\mathbb{R}}$}.  
\]
Then the action of $L$ on $X$ is free, 
and each $L$-orbit is closed.  
However, 
 the action is not proper.  
To see this, 
 consider the compact subset of $X$, 
 defined by $S:=\{(x,y):x^2+y^2=1\}$.  
Then $L_S=L$, showing that the $L$-action fails to be proper.

Moreover, 
the two points $(0,1)$ and $(1,0)$ define different points
 in the quotient space $L \backslash X$, 
 however, 
 these two points cannot be separated
 by open sets in the quotient topology.  
Hence, $L \backslash X$ is not Hausdorff. 
\end{example}

In the next section, 
 we show how the setting of Example~\ref{ex:R2b} naturally
 arises from the framework 
 of triples $L \subset G \supset H$
 of Lie groups.

\subsection{Group Theoretic Viewpoint: Properness for Triples $(L, G, H)$}
\label{subsec:properLGH}
~~~\newline
Let $G$ be a locally compact group, 
 and consider a triple of locally compact groups
\[
   L \subset G \supset H, 
\]
 where $L$ and $H$ are closed subgroups.  
We consider the natural action of the subgroup $L$
 on the homogeneous space
 $X:=G/H$.

\begin{example}
\label{ex:R2b}
Let $G:=SL(2,{\mathbb{R}})$.  
We define two subgroups of $G$
 by 
\begin{align*}
  A:=&\{\begin{pmatrix} e^t & 0 \\ 0 & e^{-t} \end{pmatrix}:
t \in {\mathbb{R}} \}, 
\\
  N:=&\{\begin{pmatrix} 1 & n \\ 0 & 1 \end{pmatrix}:
n \in {\mathbb{R}}\}.  
\end{align*}

There is a natural diffeomorphism
\[
G/N \simeq {\mathbb{R}}^2 \setminus \{(0, 0)\}, \quad
g N \mapsto g \begin{pmatrix} 1 \\ 0 \end{pmatrix}, 
\]
under which the $A$-action on $G/N$
 coincides with that described in Example~\ref{ex:R2a}.  
Hence, 
 the $A$-action on $G/N$ is not proper.  
By duality---as stated in Proposition~\ref{prop:LH_proper} below---the $N$-action on $G/A$ is likewise not proper.  
\end{example}


\subsection{Lipsman's Conjecture (1995)}
\label{subsec:Lipsman}
~~~\newline
Suppose
 that $L \subset G \supset H$ is a triple of {\emph{real reductive Lie groups}}.  
In 1989,
 the author obtained a criterion
 for the properness of the $L$-action on $G/H$.  
The result, 
 originally proved in \cite{K89}, 
 can be reformulated below, 
 see also Theorem~\ref{thm:proper89}:
\begin{theorem}
[{\cite[Example 5 (1)]{KICM90sate}}]
\label{thm:proper}
The following two conditions are equivalent:
\newline
{\rm{(i)}}\enspace
the $L$-action on $G/H$ is proper;
\newline
{\rm{(ii)}}\enspace
the $L$-action on $G/H$ has the (CI) property
 (Definition~\ref{def:CI}).  
\end{theorem}

In 1995, 
 Lipsman \cite{Li95} raised a question
 of whether the equivalence in Theorem \ref{thm:proper} remains valid
 for triples of {\emph{nilpotent Lie groups}}.  

\begin{conjecture}
[Lipsman's Conjecture {\cite{Li95}}]
Let $G$ be a connected and simply connected nilpotent Lie group, 
 and let $L$, $H$ be two connected closed subgroups.  
Are the following two conditions equivalent?
\newline
{\rm{(i)}}\enspace
the $L$-action on $G/H$ is proper;
\newline
{\rm{(ii)}}\enspace
the $L$-action on $G/H$ has the (CI) property.  
\end{conjecture}

The implication (i) $\Rightarrow$ (ii) holds in general
 (see Proposition~\ref{prop:118} (2)).  
On the other hand, 
 in the simply-connected nilpotent setting, 
 condition (ii) is equivalent to:
\begin{itemize}
\item[{\rm{(ii)$'$}}]
{\sl{the $L$-action on $G/H$ is free.}}  
\end{itemize}

\vskip 1pc
Lipsman's conjecture has been proved affirmatively 
 when the nilpotent Lie group $G$ is at most 3-step;
 that is, 
 when 
\[
  [{\mathfrak{g}}, [{\mathfrak{g}}, [{\mathfrak{g}}, {\mathfrak{g}}]]]=\{0\}, 
\]
but it fails in general.  
The status is summarized as follows:

\begin{align*}
\text{True}:\enspace
&\text{$G$:\enspace 2-step nilpotent Lie groups (Nasrin \cite{na01}), }
\\
\text{}
&\text{$G$:\enspace 3-step nilpotent Lie groups (Baklouti--Khlif \cite{BaKh05}, Yoshino \cite{xyoshino04}), }
\\
\text{False}:\enspace
&\text{$G$:\enspace 4-step nilpotent Lie groups (Yoshino \cite{xyoshino05}).}
\end{align*}

A counterexample discovered by Yoshino
 is given by 
 a triple of simply connected nilpotent Lie groups
$L \subset G \supset H$
 such that 
\begin{itemize}
\item[]
 $L \simeq {\mathbb{R}}^2$ (abelian subgroup), 
\item[]
 $X=G/H$ is a 5-dimensional nilmanifold $\simeq {\mathbb{R}}^5$, 
\end{itemize}
 with the following properties:
\par\indent
--- the $L$-action on $X=G/H$ is free;
\par\indent
--- all $L$-orbits are closed;
\par\indent
--- the orbit space $L \backslash X$ is a (T$_1$) space but not Hausdorff;
\par\indent
--- the $L$-action on $X$ is not proper.

\section{Properness Criterion}
\label{sec:proper}
In this section, 
 we present criteria for the properness of the $L$-action
 on the homogeneous space $G/H$, 
 where $L$ and $H$ are closed subgroups
 of a Lie group $G$.

As shown in Section~\ref{subsec:Lipsman}, 
 where $G$ is a simply-connected
 3-step nilpotent Lie group, 
 the (CI) property provides a convenient necessary condition
 for the properness of the action.

Here, 
 we focus on the case where $G$ is reductive.

A perspective put forward in \cite{K89} 
 concerning the properness criterion 
 for the $L$-action on $G/H$ emphasizes
 that $L$ and $H$ should be regarded symmetrically within the group $G$, 
 rather than relying on the geometric features 
 of the homogeneous space $G/H$, 
 as in some prior approaches.

To further articulate this idea,
 in Section~\ref{subsec:pitch_sim}, 
 we recall the binary relations
 $\pitchfork$ and $\sim$
 on the power set of a locally compact group $G$, 
 which were introduced in \cite{K96}
 as a conceptual framework 
 for understanding \lq\lq{the geometry at infinity}\rq\rq\
 in the group $G$ itself, 
 rather than attempting to understand
 \lq\lq{the geometry at infinity}\rq\rq\
 of the homogeneous space $G/H$.

\subsection{Expanding a Subset $H$ of a Group $G$ by Compact Set $S$}
~~~\newline
Let $H$ be a subset of a locally compact group $G$, 
 let $S \subset G$ be a compact subset.  
We define the expansion of $H$
 by $S$ through group multiplication as follows.  
\begin{align*}
S H :=&\{b x : x \in H, b \in S\}.  
\\
H S:=&\{x b : x \in H, b \in S\}, 
\\
S H S:=&\{a x b : x \in H, a, b \in S\}.  
\end{align*}

When $G$ is abelian, 
 the subsets $S H$ and $H S$ may be thought of as tubular neighbourhoods
 of $H$, 
 as well as the subset $SHS=(SS)H=H(SS)$.  
In contrast, 
 when $G$ is highly non-commutative---such as
in the case of $SL(2,{\mathbb{R}})$---the set $S H S$
 can become significant \lq\lq{larger}\rq\rq\
 in a nontrivial way.  
A deeper understanding of the structure of $S H S$ 
 allows us to reformulate
 the problem of properness of the $L$-action on $G/H$
 as a question internal to the group $G$ itself.

Here is a straightforward and fundamental observation:

\begin{lemma}
\label{lem:pitchpro}
Suppose that both $L$ and $H$ are closed subgroups of $G$.  
Then the following two conditions on the pair $(L, H)$ are equivalent:

{\rm{(i)}}\enspace
The action of $L$ on $G/H$ is proper.

{\rm{(ii)}}\enspace
For every compact subset $S \subset G$, 
 the intersection $L \cap SHS$ is compact.  
\end{lemma}

\begin{proof}
Let $S$ be a compact subset of $G$, 
 and let $\overline S:= SH/H \subset G/H$.  
Then $\overline S$ is compact.  
Conversely, 
 every compact subset of $G/H$ can be expressed
 in this form 
 for some compact subset $S \subset G$.  
By the definition of proper actions, 
 condition (i) is equivalent to 
 the compactness of the set
\[
   L_{\overline S}:=\{\ell \in L: \ell \cdot \overline S \cap \overline S \ne \emptyset\}
\]
 for every compact subset $S \subset G$.

Without loss of generality, 
 we assume 
 that $S$ is symmetric, 
 {\it{i.e.,}} $S=S^{-1}$.  
Under this assumption, 
 one has 
\[
   L_{\overline S} = L \cap S H S^{-1} =L \cap SHS, 
\]
 which shows the equivalence (i) $\Leftrightarrow$ (ii).  
\end{proof}


\subsection{\protect$\pitchfork$ and \protect$\sim$
 for locally compact groups $G$}
\label{subsec:pitch_sim}
~~~\newline
Let ${\mathcal{P}}(G)$ denote
 the power set of a locally compact group $G$.  
We define the binary relations $\pitchfork$ and $\sim$
 on the power set ${\mathcal{P}}(G)$ of the group $G$.

\begin{definition}
[{\cite[Def.\ 2.1.1]{K96}}]
\label{def:pitch}
For two subsets $L$ and $H$ of $G$, 
 we define the following two binary relations
 $\pitchfork$ and $\sim$:
\newline
(1)\enspace
$L \pitchfork H$ if for every compact subset $S \subset G$, 
 the intersection $L \cap SHS$ is relatively compact, 
 {\it{i.e.}}, 
 its closure is compact;
\newline
(2)\enspace
$L \sim H$ if there exists a compact subset $S \subset G$
 such that both $L \subset SHS$ and $H \subset SLS$.
\end{definition}

We illustrate these definitions with simple examples:

\begin{example}
[Abelian Case]
\label{ex:pitch}
Let $G$ be a vector space $\mathbb R^n$, 
 and let $L$, $H$ be subspaces of $G$.  
\newline
{\rm{(1)}}\enspace
$L \pitchfork H$  if and only if $L \cap H = \{0\}$.  
\newline
{\rm{(2)}}\enspace
$L \sim H$  if and only if $L = H$. 
\end{example}

\begin{example}
Let $G=SL(2,{\mathbb{R}})$.  
Up to conjugation, 
 there are six connected subgroups of $G$:  
\[
  \{e\}, K, A, N, AN, \text{ and } G, 
\]
where $A$ and $N$ are as defined
 in Example~\ref{ex:R2b}, 
 and $K=SO(2)$.  
We then have
\begin{align*}
  &\{e\} \sim K, 
\\
  & A \sim N \sim A N \sim G.  
\end{align*}
\end{example}

This $SL_2$ example can be generalized
 in two directions as follows:

\begin{example}
\label{ex:similar_A_N}
Let $G$ be a real reductive linear group.  
Then the following decompositions hold:
\begin{alignat*}{2}
G=& K A K \qquad
&&\text{Cartan decomposition, \quad see \eqref{eqn:KAK}, }
\\
G=& K A N \qquad
&&\text{Iwasawa decomposition, }
\\
G=&KNK.  
&&
\end{alignat*}
Hence, 
 we have
\begin{align*}
& \{e\} \sim K, 
\\
& A \sim N \sim AN \sim G.  
\end{align*}
\end{example}

\begin{example}
Let $G$ be a real reductive Lie group
 of split rank one.  
It follows that
 for any closed subgroup $L$, 
 not necessarily connected, 
 either $L \sim \{e\}$ or $L \sim G$ holds
 ({\it{cf.}} \cite{Ko93}).  
\end{example}

\vskip 3pc

\subsection{Meaning of the Binary Relations \protect$\pitchfork$ and \protect$\sim$}
~~~\newline
In Definition~\ref{def:pitch}, 
 $L$ and $H$ are allowed to be {\textit{subsets}} of $G$, 
 without assuming that they are closed subgroups.

The following two lemmas can be verified
 directly from Definition~\ref{def:pitch}.

\begin{lemma}
\label{lem:sim_equiv}
The relation $\sim$ defines an equivalence relation on ${\mathcal{P}}(G)$;
 that is, 
 for any $L_1, L_2, L_3 \in {\mathcal{P}}(G)$, 
 the following properties hold
 in addition to the obvious reflexivity: 
\newline
{\rm{(1)}}\,{\rm{(symmetry)}}\enspace
$L_1 \sim L_2$ if and only if $L_2 \sim L_1$;
\newline
{\rm{(2)}}\,{\rm{(transitivity)}}\enspace
if $L_1 \sim L_2$ and $L_2 \sim L_3$, 
 then $L_1 \sim L_3$.  
\end{lemma}

\begin{lemma}
\label{lem:simpitch}
Let $H$, $H'$, and $L$ be subsets of a locally compact group $G$.  
\newline
{\rm{(1)}}\enspace
$L \pitchfork H$  if and only if $H \pitchfork L$.  
\newline
{\rm{(2)}}\enspace
If $H \sim H'$, 
 then the following equivalence holds for any $L$:
\[
   H \pitchfork L \iff H' \pitchfork L.  
\]
\end{lemma}

As an immediate consequence of Lemma~\ref{lem:pitchpro}
 and the definition of $\pitchfork$
 in Definition~\ref{def:pitch}, 
 we have the following proposition.  
\begin{proposition}
Let $L$ and $H$ be closed subgroups
 of a locally compact group $G$.  
Then the relation $L \pitchfork H$ holds
 if and only if the action of $L$ on $G/H$ is proper.  
\end{proposition}

Lemma~\ref{lem:simpitch} (1) clarifies
 the symmetry between the closed subgroups $L$ and $H$ in $G$
 with respect to proper actions:
\begin{proposition}
\label{prop:LH_proper}
Let $L$ and $H$ be closed subgroups 
 of a locally compact group $G$.  
Then, the action of $L$ on $G/H$ is proper
 if and only if the action of $H$ on $G/L$ is proper.  
\end{proposition}

In light of Lemma~\ref{lem:simpitch} (2), 
 we can formulate the following fundamental problem 
 concerning a criterion for properness 
 in the general framework, 
 as follows:

\begin{problem}
[Properness Criterion]
\label{prob:pitchfork}
Find a criterion 
 for two subsets $L, H \subset G$
 to satisfy 
\[
L \pitchfork H, 
\]
modulo the equivalence relation $\sim$.  
\end{problem}

As we shall see in Theorem~\ref{thm:Ddual} below, 
 the equivalence relation $\sim$ on ${\mathcal{P}}(G)$
 is the coarsest equivalence relation 
 that preserves the properness condition $\pitchfork$.

\vskip 3pc

\subsection{Discontinuous Duality Theorem}
~~~\newline
For a subset $H$ of a locally compact group $G$, 
 we define its \lq\lq{discontinuous dual}\rq\rq\
 by 
\[
   \pitchfork(H:G) := \{L \in {\mathcal{P}}(G): \text{$L \pitchfork H$}\}.  
\]
The discontinuous dual $\pitchfork(H:G)$ depends solely 
 on the equivalence class of $H$ under the relation $\sim$, 
 as stated in Lemma~\ref{lem:sim_equiv}.  
Inspired by the Pontrjagin--Tannaka--Tatsuuma duality theorem {\cite{tatsu}}, 
 which roughly states 
 that a locally compact group $G$ 
 can be recovered from its unitary dual $\widehat G$, 
 the present author suggested in \cite[Thm.\ 5.6]{K96}
 a \lq\lq{discontinuous duality theorem}\rq\rq\
 formulated as follows.

\begin{theorem}
[Discontinuous Duality Theorem]
\label{thm:Ddual}
Let $G$ be a separable, locally compact topological group.  
Then any subset $H \subset G$ is determined, 
 up to the equivalence relation $\sim$, 
 by its discontinuous dual $\pitchfork(H:G)$.  
\end{theorem}

Theorem~\ref{thm:Ddual} was first proved 
 for real reductive Lie groups $G$ in \cite{K96}, 
 and was later extended 
 to general locally compact groups 
 by Yoshino \cite{xyoshino07}.

\vskip 3pc


\subsection{Properness Criterion for Reductive Groups}
~~~\newline
It is worth emphasizing 
 that by the term \lq\lq{criterion}\rq\rq, 
 we mean an {\emph{explicit and effective}} method 
 for determining 
 whether the relation $L \pitchfork H$ holds---not merely
 a theoretically correct
 but practically intractable reformulation.  
In this context, 
 Problem~\ref{prob:pitchfork}
 remains open
 for general Lie groups.  
However, 
in the case 
 where $G$ is a real reductive Lie group, 
 the problem has been resolved, 
 as reviewed 
 in Theorem~\ref{thm:proper96} below.

Let $G$ be a real reductive group, 
 ${\mathfrak{g}}={\mathfrak{k}}+{\mathfrak{p}}$ 
 a Cartan decomposition
 of its Lie algebra, 
 and $K$ the maximal compact subgroup of $G$
 with Lie algebra ${\mathfrak{k}}$.  
We take a maximal abelian subspace ${\mathfrak{a}}$ in ${\mathfrak{p}}$.

For $\alpha \in {\mathfrak{a}}^{\ast}$, 
 we define the root space
 and the set of (restricted) roots by 
\[
   {\mathfrak{g}}_{\alpha}:=
 \{X \in {\mathfrak{g}}:
\text{$\operatorname{ad}(H) X = \alpha (H)$ for all $H \in {\mathfrak{a}}$} \}, 
\]
\[
  \Sigma({\mathfrak{g}}, {\mathfrak{a}})
  :=\{\alpha \in {\mathfrak{a}}^{\ast}:
      {\mathfrak{g}}_{\alpha} \ne \{0\}
  \}.  
\]
Then the finite group
\[
  N_G({\mathfrak{a}})/Z_G({\mathfrak{a}})
  =
  \{g \in G: \operatorname{Ad}(g) {\mathfrak{a}}={\mathfrak{a}}\}
  \big/
  \{g \in G: \operatorname{Ad}(g)|_{\mathfrak{a}}=\operatorname{id}\}
\]
 is isomorphic to the Weyl group, 
 to be denoted by $W$, 
 of the restricted root system
 $\Sigma(\mathfrak{g},\mathfrak{a})$.

We fix a set of positive roots $\Sigma^+({\mathfrak{g}}, {\mathfrak{a}})$, 
 and define the (closed) dominant Weyl chamber
 $\overline{{\mathfrak{a}}_+}$
 by 
\[
   \overline{{\mathfrak{a}}_+}
   :=\{H \in {\mathfrak{a}}:\alpha(H)\ge 0\text{ for all $\alpha \in \Sigma^+({\mathfrak{g}}, {\mathfrak{a}})$}\}.  
\]
Then we have a natural bijection 
\begin{equation}
\label{qn:aWquotient}
\overline{{\mathfrak{a}}_+} \simeq {\mathfrak{a}}/W.  
\end{equation}
We set
 $A:=\exp({\mathfrak{a}})$
 and $\overline{A_+}:=\exp(\overline{{\mathfrak{a}}_+})$.

In analogy with polar coordinates
 in the Euclidean space ${\mathbb{R}}^n$, 
 there exists a notion of polar coordinates
 on the Riemannian symmetric space $G/K$.  
In group-theoretic terms, 
 this corresponds to the fact 
 that a real reductive Lie group $G$ admits a Cartan decomposition:

\begin{equation}
\label{eqn:KAK}
G = K A K = K \overline{A_+}K.  
\end{equation}

In this decomposition, 
 every $g \in G$ can be written
 as $g \in K \exp (\mu(g)) K$, 
 where $\mu(g) \in {\mathfrak{a}}$ is unique up to conjugation
 by the Weyl group $W$.

The Cartan decomposition \eqref{eqn:KAK} defines the {\emph{Cartan projection}}:
\begin{equation}
\label{eqn:Cartanpr}
   \mu \colon G \to {\mathfrak{a}}/W \simeq \overline{{\mathfrak{a}}_+}, \quad
   g \mapsto H \mod W, 
\end{equation}
characterized by the condition that $g \in K \exp(H) K$.

\begin{example}
\label{ex:GLCartan}
Let $G=GL(n,{\mathbb{R}})$ and $K=O(n)$.  
Then we can identify 
\[
 \text{${\mathfrak{a}}\simeq {\mathbb{R}}^n$, 
 $W\simeq {\mathfrak{S}}_n$, 
 and $\overline{{\mathfrak{a}}_+} \simeq {\mathbb{R}}_{\ge}^n 
:=\{(H_1, \dots, H_n): H_1 \ge \cdots \ge H_n\}$.  
}
\]
If $g=k_1 \operatorname{diag}(e^{H_1}, \dots, e^{H_n}) k_2$
 for some $k_1, k_2 \in O(n)$
 and $H_1, \dots, H_n \in {\mathbb{R}}$, 
 then 
\[
\text{${}^{t\!} g g = {}^{t\!}k_2 \operatorname{diag}(e^{2H_1}, \dots, e^{2H_n})k_2$. }
\]
Hence, 
 the Cartan projection 
\[
\text{$\mu \colon GL(n,\mathbb{R}) \to \mathbb R^n/{\mathfrak{S}}_n \simeq {\mathbb{R}}_{\ge}^n$}
\]
 is given by 
\[
 g \mapsto \tfrac 1 2 (\log \lambda_1,\cdots, \log \lambda_n),  
\]
where $\lambda_1 \ge \dots \ge \lambda_n\,(>0)$ are the eigenvalues of ${}^tgg$.\end{example}

The following properness criterion was established
 by Benoist \cite[Thm.\ 5.2]{B96}
 and the present author \cite[Thm.\ 1.1]{K96}, 
 extending the criterion given in \cite{K89}
 for the special case 
 where $L$ and $H$ are reductive subgroups 
 (see Theorem~\ref{thm:proper89}).  

\begin{theorem}
[Properness Criterion]
\label{thm:proper96}
Let $G$ be a reductive Lie group, 
 and let $H$, $L$ be subsets of $G$.  
Then the following equivalences hold:

\begin{tabular}{llll}
{\rm{(1)}} & 
$L \sim H$ in $G$ & 
$\Longleftrightarrow$ & 
$\mu(L) \sim \mu(H)$ in $\mathfrak a$. \\
{\rm{(2)}} &
$L \pitchfork H$ in $G$ & 
$\Longleftrightarrow$ & 
$\mu(L) \pitchfork \mu(H)$ in $\mathfrak a$. 
\end{tabular}
\end{theorem}

\begin{example}
Let $G=SL(2,{\mathbb{R}})$, $L:=A$.  
Then the Cartan projection
 $\mu \colon {\mathfrak{g}} \to {\mathfrak{a}}/W \simeq \overline{{\mathfrak{a}}_+}$
 gives
\[
   \mu(A)=\mu(N)=\overline{{\mathfrak{a}}_+}.  
\]
Hence, $A \not \pitchfork N$.  
We have seen this directly
 in Example~\ref{ex:R2b}, 
 which describes the ${\mathbb{R}}$-action
 on ${\mathbb{R}}^2 \setminus \{(0,0)\}$.  
\end{example}

It is worth noting 
 that, 
 in the equivalences in Theorem~\ref{thm:proper96}, 
 the left-hand sides are formulated
 in the non-commutative group $G$,
 whereas the right-hand sides
 are described 
  in the abelian space ${\mathfrak{a}}$.  
As a consequence, 
 the conditions $\pitchfork$ and $\sim$ can be verified
 on the right-hand sides, 
 once the Cartan projections $\mu(L)$ and $\mu(H)$ are known.

Special cases of Theorem~\ref{thm:proper96} include the following:
\begin{itemize}
\item[]
\begin{itemize}
\item[$\Rightarrow$ in (1): ]
Uniform error estimates of eigenvalues under matrix perturbation.
\item[$\Leftrightarrow$ in (2): ]
A criterion for proper actions of groups.  
\end{itemize}
\end{itemize}

Moreover, 
 in connection with the criterion for $\pitchfork$
 in (2), 
 we discuss
  in Section~\ref{sec:quantify}
 two approaches for quantifying them.

\begin{remark}
[$\sim$ and $\pitchfork$ in ${\mathfrak{a}}/W \simeq \overline{{\mathfrak{a}}_+}$]
\label{rem:25050507}
The Cartan projections $\mu(L)$ and $\mu(H)$, 
 given in \eqref{eqn:Cartanpr}, 
 can be interpreted
 in two ways:
 as subsets of $\overline{\mathfrak{a}_+}$
 or as $W$-invariant subsets of $\mathfrak{a}$.  
In either interpretation, 
 the relations $\sim$ and $\pitchfork$ between $\mu(L)$ and $\mu(H)$
 retain the same meaning.  
In fact, 
 for two subsets $S$, $T$ $\subset \overline{{\mathfrak{a}}_+}$,
 define their $W$-invariant extensions of ${\mathfrak{a}}$
 by 
\[
  \widetilde S:= W \cdot S, \quad
  \widetilde T:= W \cdot T.  
\]
Then it is readily seen from the definitions 
that the following equivalences hold:
\begin{align*}
  S \sim T &\Leftrightarrow \widetilde S \sim \widetilde T, 
\\
  S \pitchfork T &\Leftrightarrow \widetilde S \pitchfork \widetilde T.  
\end{align*}
\end{remark}

\subsection
{Properness Criterion---Special Case (Reductive Subgroups)}
~~~\newline
In this section, 
 we illustrate the idea
 behind the proof of the properness criterion
 (Theorem~\ref{thm:proper96})
 in the special case
 where $L$ and $H$ are reductive subgroups of $G$.

Since the properness criterion 
 is invariant under conjugation 
of $L$ and $H$, 
 we may, 
 without loss of generality, 
 assume that both are stable 
 under a Cartan involution $\theta$ of $G$.

To treat $L$ and $H$ in a uniform manner, 
 we introduce a $\theta$-stable subgroup $G'$, 
 and set up the corresponding notation.

Let ${\mathfrak{g}}={\mathfrak{k}}+{\mathfrak{p}}$
 be the Cartan decomposition
 corresponding to the Cartan involution $\theta$.  
Since $G'$ is $\theta$-stable, 
 its Lie algebra ${\mathfrak{g}}'$ admits a compatible Cartan decomposition
 ${\mathfrak{g}}'={\mathfrak{k}}'+{\mathfrak{p}}'$
 with ${\mathfrak{k}}' \subset {\mathfrak{k}}$
 and ${\mathfrak{p}}' \subset {\mathfrak{p}}$.  
Let ${\mathfrak{a}}_{G'}$ be a maximal abelian subspace 
 in ${\mathfrak{p}}'$, 
 which we extend to a maximal abelian subspace ${\mathfrak{a}}$
 in ${\mathfrak{p}}$.  
Then ${\mathfrak{a}}_{G'}={\mathfrak{a}} \cap {\mathfrak{g}}'$.  
We summarize these subspaces as follows.

\begin{alignat}{56}
&{\mathfrak{g}}=
&&{\mathfrak{k}}\,\,\,+\,\,
&&{\mathfrak{p}}
&&\supset\,\,
&&{\mathfrak{p}}
\underset{\text{\BLUE{\footnotesize{maximal abelian}}}}{\supset}
&&{\mathfrak{a}}
\notag
\\
&\cup
&&\cup
&&\cup
&&
&&\cup
&&\cup
\notag
\\
\label{eqn:aG}
&{\mathfrak{g}}'=
&&{\mathfrak{k}}'\,\,+\,\,
&&{\mathfrak{p}}'
&&\supset\,\,
&&{\mathfrak{p}}'
\underset{\text{\BLUE{\footnotesize{maximal abelian}}}}{\supset}
&&{\mathfrak{a}}_{G'}
:={\mathfrak{a}} \cap {\mathfrak{g}}'.  
\end{alignat}

Let $A_{G'}:=\exp({\mathfrak{a}}_{G'})$.  
By the Cartan decomposition, 
 we have $G' \sim A_{G'}$ in $G'$
 (see Example~\ref{ex:similar_A_N}), 
 and hence also $G' \sim A_{G'}$ in $G$.  
The Cartan projection of $G'$ takes the form
 $\mu(G')=W \cdot {\mathfrak{a}}_{G'}$
 in ${\mathfrak{a}}$.

Applying the above notation
 to $G'=L$, 
 we have ${\mathfrak{a}}_{L} = {\mathfrak{a}} \cap {\mathfrak{l}}$.  
By conjugating $H$ by an element of $K$, 
 we may assume ${\mathfrak{a}}_{H}:={\mathfrak{a}} \cap {\mathfrak{h}}$ is
 also a maximal abelian subspace of ${\mathfrak{h}} \cap {\mathfrak{p}}$.  
Then, 
 the condition $\mu(L) \pitchfork \mu(H)$ in ${\mathfrak{a}}$, 
 as stated in Theorem~\ref{thm:proper96} (2), 
 is equivalent 
 to the condition
\[
   {\mathfrak{a}}_H \cap W \cdot {\mathfrak{a}}_L=\{0\}.  
\] 
Hence, 
 the special case of Theorem~\ref{thm:proper96} yields the following result:

\begin{theorem}
[Properness Criterion for Reductive Subgroups {\cite{K89}}]
\label{thm:proper89}
~~~\newline
Let $G$ be a real reductive Lie group, 
 and let $H$ and $L$ be two reductive subgroups of $G$.  
Then the following conditions on the pair $(L, H)$ are equivalent:
\newline
{\rm{(i)}}\enspace
the action of $L$ on $G/H$ is proper;
\newline
{\rm{(i)}}$'$\enspace
the action of $H$ on $G/L$ is proper;
\newline
{\rm{(ii)}}\enspace
${\mathfrak{a}}_H \cap W \cdot {\mathfrak{a}}_L =\{0\}$
 in ${\mathfrak{a}}$.  
\end{theorem}

\subsection{Reduction of Properness Criterion to Abelian subgroups}
~~~\newline
We now outline the key idea from \cite{K89}
 that underlies the proof of Theorem~\ref{thm:proper89}.  
Although the proof of Theorem~\ref{thm:proper96}
 in the general case
 is more involved, 
 it is still based 
 on the same principle.

First, 
 we observe 
 that $L \sim A_L$ and $H \sim A_H$ in $G$.  
Hence, 
 the properness criterion in Theorem~\ref{thm:proper89}
 in the reductive case
 can be reduced to the abelian case, 
 where $L, H \subset A$.  
This reflects the fact
 that the ambient group $G$ itself 
 is highly non-commutative.  
This reduction is formulated as Lemma~\ref{lem:pro89} below.

\begin{lemma}
\label{lem:pro89}
Suppose that $L$ and $H$ are connected subgroups
 of the split abelian subgroup $A$.  
Let ${\mathfrak{l}}, {\mathfrak{h}} \subset {\mathfrak{a}}$
 be the (abelian) Lie algebras of $L$ and $H$, 
 respectively.  
Then the following are equivalent:
\newline
{\rm{(i)}}\enspace
The action of $L$ on $G/H$ is proper.  
\newline
{\rm{(ii)}}\enspace
 ${\mathfrak{l}} \cap W{\mathfrak{h}}=\{0\}$.  
\end{lemma}

The implication (i) $\Longrightarrow$ (ii) is the easier direction, 
 that is, 
 properness implies the (CI) property, 
 as seen in Proposition~\ref{prop:118} (3).

The converse implication (ii) $\Longrightarrow$ (i) is more involved.  

In the next section, 
 we will give an overview of the proof.

\subsection{Proof of Lemma~\ref{lem:pro89} for Abelian Subgroups $H$, $L \subset G$}
~~~\newline
Suppose that both ${\mathfrak{l}}$ and ${\mathfrak{h}}$ are subspaces
 of ${\mathfrak{a}}$.  
We aim to prove
 that if $L \not \pitchfork H$, 
 then ${\mathfrak{l}} \cap W{\mathfrak{h}} \ne \{0\}$.

If $L \not \pitchfork H$, 
 then there exists a compact subset $S \subset G$
 such that the intersection $L \cap S H S$ is non-compact.  
Hence, 
 one can find sequences $t_n$, $t_n' \in {\mathbb{R}}$, 
 $Y_n \in {\mathfrak{l}}$, 
 $Z_n \in {\mathfrak{h}}$
 with $\|Y_n\|=\|Z_n\|=1$, 
 and $c_n, d_n \in S$
such that
\begin{align*}
\exp(t_n Y_n)
=\,& c_n \exp (t_n' Z_n) d_n
\quad\text{ in $G$}, 
\\
\lim_{n \to \infty}t_n=\,&\infty.  
\end{align*}

By passing to a subsequence, 
 we may assume
 that the sequences $c_n$, $d_n$, $Y_n$, and $Z_n$
 converge
 as $n$ tends to infinity, 
 say, 
\begin{align*}
&c_n \to c,\,\, d_n \to d \text{ in }S
\\
&Y_n \to Y \ne 0 \in {\mathfrak{l}},\quad
 Z_n \to Z \ne 0 \in {\mathfrak{h}}.  
\end{align*}
\par\noindent
{\bf{Step 1.}}\enspace
We show
 that the sequence $\frac{t_n'}{t_n}$ is bounded
 away from $0$.

Once this boundedness is established, 
 we may again pass to a subsequence
 and assume
 that $\frac{t_n'}{t_n}$ converges.  
By replacing $(S, t_n', Z_n, c_n, d_n)$
 with $(KSK, t_n, \frac{t_n'}{t_n} Z_n, c_n k, k^{-1} d_n)$
 for some $k \in N_K({\mathfrak{a}})$,  
 we may further assume that $t_n'= t_n$
 and $Y, Z \in \overline{\mathfrak{a}_+}$.  
Thus, 
 we obtain sequences
 such that $\underset{n\to \infty} \lim t_n=\infty$
 and:
\begin{align}
\label{eqn:ctYd}
&c_n=\exp (t_n Y_n) d_n^{-1} \exp(-t_n Z_n),
\\
\notag
&c_n \to c,\,\, d_n \to d \text{ in }G,
\\
\notag
&Y_n \to Y \in {\mathfrak{l}} \cap \overline{\mathfrak{a}_+},\,\,
 Z_n \to Z \in {\mathfrak{h}} \cap \overline{\mathfrak{a}_+}.  
\end{align}
\par\noindent
{\bf{Step 2.}}
We now derive that $Y=Z$ from \eqref{eqn:ctYd}.

Both Steps 1 and 2 deal with the behavior of sequences
 \lq\lq{at infinity}\rq\rq\
 in the group $G$.  
To analyze the \lq\lq{geometry at infinity}\rq\rq\
 of the group $G$, 
 we localize the analysis 
 by examining the dynamics
 in terms of the root space decomposition 
${\mathfrak{g}}= \underset{\alpha \in \Sigma({\mathfrak{g}};{\mathfrak{a}}) \cup\{0\}}{\bigoplus} {\mathfrak{g}}_{\alpha}
$
 via the adjoint representation.

To illustrate the method applied 
 in both Steps, 
 we consider the following identity, 
 which plays a crucial role
 in establishing $Y=Z$ in Step 2:
\begin{equation}
\label{eqn:bYaZ}
   \operatorname{Ad}(c) {\mathfrak{g}}_{\alpha} 
 = \sum_{\beta(Y) \ge \alpha(Z)} {\mathfrak{g}}_{\beta}
\end{equation}
 for any $\alpha \in \Sigma({\mathfrak{g}}, {\mathfrak{a}})$.  
To verify \eqref{eqn:bYaZ}, 
 let $\operatorname{pr}_{\alpha} \colon {\mathfrak{g}} \to {\mathfrak{g}}_{\alpha}$
 denote the projection
 associated with the root space decomposition.  
{}From \eqref{eqn:ctYd}, 
 it follows that for any $\alpha \in \Sigma({\mathfrak{g}}, {\mathfrak{a}})$, 
\begin{align*}
\operatorname{Ad}(c) {\mathfrak{g}}_{\alpha}
=&
\lim_{n\to \infty} 
\bigoplus_{\beta\in \Sigma({\mathfrak{g}}, {\mathfrak{a}}) \cup \{0\}} 
e^{t_n
(\beta(Y_n) 
-\alpha(Z_n))}
\operatorname{pr}_{\beta}(\operatorname{Ad}(d_n^{-1}) {\mathfrak{g}}_{\alpha})
\\
\subset &\bigoplus_{\beta(Y) \ge \alpha(Z)}{\mathfrak{g}}_{\beta}.  
\end{align*}
The opposite inclusion follows similarly, 
 and thus \eqref{eqn:bYaZ} holds.  
This argument implies $Y=Z$, 
 hence ${\mathfrak{l}} \cap {\mathfrak {h}} \ne \{0\}$.  
See \cite{K89} for further details.  
\qed

\subsection{Criterion for the Calabi--Markus Phenomenon}
~~~\newline
The Calabi--Markus phenomenon (Theorem~\ref{thm:CM}), 
 originally discovered in \cite{CM}
 in the context of the de Sitter space, 
 can be formulated in a more general setting as follows.  

\begin{corollary}
[Criterion for the Calabi--Markus Phenomenon {\cite{K89}}]
\label{cor:CMcri}
Let $G \supset H$ be a pair of real reductive Lie groups.  
Then the following four conditions (i)---(iv) are equivalent:
\newline
{\rm{(i)}}\enspace
$G/H$ admits a discontinuous group $\Gamma \simeq {\mathbb{Z}}$.  
\newline
{\rm{(ii)}}\enspace
$G/H$ admits an infinite discontinuous group $\Gamma$.  
\newline
{\rm{(iii)}}\enspace
$G \not \sim H$. 
\newline
{\rm{(iv)}}\enspace
$\operatorname{rank}_{\mathbb{R}} G>\operatorname{rank}_{\mathbb{R}}H$.  
\end{corollary}

The original result by Calabi and Markus
 in \cite{CM}
 shows that condition (ii) fails to hold
 when $(G, H)=(O(n,1), O(n-1,1))$.  
In this case, 
we observe 
 $\operatorname{rank}_{\mathbb{R}} G = \operatorname{rank}_{\mathbb{R}}H=1$, 
 and consequently condition (iv) fails as well.

\begin{proof}
The implication (i) $\Rightarrow$ (ii) is immediate.  
\newline
(ii) $\Rightarrow$ (iii).\enspace
Suppose that (ii) holds.  
Then $\Gamma \pitchfork H$.  
However, 
 since $\Gamma \not \pitchfork G$, 
 Lemma~\ref{lem:simpitch} (2) implies
 that $H \not \sim G$.  
Thus, 
 (ii) $\Rightarrow$ (iii) is verified.  
\newline
(iii) $\Rightarrow$ (iv).\enspace
Without loss of generality, 
 assume
 that ${\mathfrak{a}}_H \subset {\mathfrak{a}}$
 as in \eqref{eqn:aG}.  
Since 
$
   \operatorname{rank}_{\mathbb{R}} H= \dim {\mathfrak{a}}_H
$, 
 the assumption $\operatorname{rank}_{\mathbb{R}}G = \operatorname{rank}_{\mathbb{R}}H$ implies 
 $\mu(H)={\mathfrak{a}}$.  
Hence, 
 by the easier direction
 of Theorem~\ref{thm:proper89}, 
 we have $G \sim H$.  
This completes the proof of (iii) $\Rightarrow$ (iv).  
\newline
(iv) $\Rightarrow$ (i).\enspace
This is the main step.  
If $
   \operatorname{rank}_{\mathbb{R}} G > \operatorname{rank}_{\mathbb{R}}H
$, 
 {\it{i.e.}}, 
 if ${\mathfrak{a}}_H \subsetneqq {\mathfrak{a}}$, 
 then there exists a one-dimensional subspace ${\mathfrak{l}}$
 in ${\mathfrak{a}}$
 such that $W \cdot {\mathfrak{a}}_H \cap {\mathfrak{l}} =\{0\}$.  
By the properness criterion 
in Theorem~\ref{thm:proper89}, 
 the subgroup $L:= \exp {\mathfrak{l}}$ acts properly on $G/H$.  
In particular, 
 any lattice in $L$, isomorphic to ${\mathbb{Z}}$ acts 
 properly discontinuously on $G/H$.  
\end{proof}

\subsection{Proper Actions of $SL(2,{\mathbb{R}})$}
~~~\par
In the previous section, 
 we discussed proper actions
 of \emph{commutative} subgroups
 on homogeneous spaces.  
Here, we turn to the case of \emph{non-commutative} subgroups, 
 illustrating the discussion 
 with the example of $SL(2,{\mathbb{R}})$
 or $PSL(2,{\mathbb{R}})$.

\begin{proposition}
\label{prop:surface_gp}
Let $G$ be a real reductive linear Lie group, 
 and let $H$ be a closed subgroup
 (possibly non-reductive, {\it{e.g.}}, a discrete subgroup).  
Consider the following five conditions:
\newline
{\rm{(i)}}\enspace
$G/H$ admits a discontinuous group 
 $\Gamma \simeq {\mathbb{Z}}$
 generated
 by a unipotent element.  
\newline
{\rm{(ii)}}\enspace
$G/H$ admits a proper action of a subgroup $L$ 
 which is locally isomorphic to $SL(2,{\mathbb{R}})$.  
\newline
{\rm{(iii)}}\enspace
For any $g \ge 2$, 
 $G/H$ admits a discontinuous group $\Gamma$ 
 isomorphic to $\pi_1(\Sigma_g)$, 
 where $\Sigma_g$ is a closed oriented surface 
 of genus $g$. 
\newline
{\rm{(iv)}}\enspace
For some $g \ge 2$, 
$G/H$ admits a discontinuous group 
 $\Gamma\simeq \pi_1(\Sigma_g)$.  
\newline
{\rm{(v)}}\enspace
$G/H$ admits a discontinuous group $\Gamma$
 of infinite order, 
 which is not virtually abelian, 
 {\it{i.e.,}}
 $\Gamma$ does not contain an abelian subgroup of finite index.  

\par
Then the following implications and equivalences hold:
\[
(i) \Leftrightarrow (ii) \Rightarrow (iii) \Rightarrow (iv) \Rightarrow (v).  
\]
\end{proposition}

\begin{proof}
The equivalence (i) $\Leftrightarrow$ (ii)
 ({\it{cf.}} \cite[Lem.\ 3.2]{Ko93})
 follows from the Jacobson--Morozov theorem.

Since any surface group can be embedded
 as a discrete subgroup 
 of $PSL(2,{\mathbb{R}})$, 
 and also of $SL(2,{\mathbb{R}})$,
 the implication (ii) $\Rightarrow$ (iii) follows.  

The remaining implications (iii) $\Rightarrow$ (iv) $\Rightarrow$ (v)
 are straightforward.  
\end{proof}

\subsection{An Example: Actions of $SL(2,{\mathbb{R}})$ on $SL(n,{\mathbb{R}})/SL(m,{\mathbb{R}})$}
~~~\par
In the previous section, 
 we discussed general properties
 of non-abelian groups
such as surface groups
 and $SL(2,{\mathbb{R}})$
 on homogeneous spaces.  
In this section, 
 we examine properness of $SL(2,{\mathbb{R}})$-actions
 more concretely through an explicit example.  
Specifically, 
 we consider the action of $SL(2,{\mathbb{R}})$
 on the homogeneous space $G/H$
 via a group homomorphism
\[
 \varphi \colon SL(2,{\mathbb{R}}) \to G.  
\]
There are, 
in fact, many such homomorphisms $\varphi$, 
and the properness of the induced action generally 
 depends one the choice of $\varphi$.  
We illustrate this dependence with the case, 
 where $G=SL(n,{\mathbb{R}})$
 and $H=SL(m,{\mathbb{R}})$ is the subgroup
 embedded block-diagonally in $G$ with $m<n$.  

\begin{question}
Suppose that 
 $\varphi_n \colon SL(2,{\mathbb{R}}) \to SL(n,{\mathbb{R}})$
 is an irreducible representation.  
Is the action of $SL(2,{\mathbb{R}})$ on $G/H$
 via $\varphi_n$ proper?
\end{question}

See also Example~\ref{ex:4.7} (1) below
 for a related discussion
 on the existence problem
 of cocompact discontinuous groups
 for the same homogeneous space $G/H$.

\vskip 1pc
Let $L:=\varphi_n(SL(2,{\mathbb{R}}))$.  
To apply Theorem~\ref{thm:proper89}, 
 we compute $\mu(H)$ and $\mu(L)$.  
We define a maximal abelian subspace ${\mathfrak{a}}$
 by the diagonal embedding
\[
{\mathfrak{a}}:=\{(a_1, \cdots, a_n):\sum_{j=1}^n a_j=0\}
\underset{\operatorname{diag}}{\hookrightarrow} 
{\mathfrak{g}}= {\mathfrak{s l}}(n,{\mathbb{R}}).  
\]
Then the Cartan projection is given by 
$
   \mu \colon G \to {\mathfrak{a}}/{\mathfrak{S}}_n
$
for $G=SL(n,{\mathbb{R}})$.

For $H=SL(m,{\mathbb{R}})$ in $G=SL(n,{\mathbb{R}})$\,\, ($m<n$), 
\[
\mu(H)={\mathfrak{S}}_n \cdot {\mathfrak{a}}_H
={\mathfrak{S}}_n
 \cdot \{(b_1, \cdots,b_m,0,\cdots, 0):\sum_{j=1}^{m}b_j=0\}.  
\]
On the other hand, 
 for the irreducible representation $\varphi_n$, 
 we have 
\[
  \mu(L)={\mathfrak{S}}_n \cdot {\mathfrak{a}}_L
        ={\mathfrak{S}}_n \cdot {\mathbb{R}} (n-1,n-3,\cdots,1-n).  
\]

By the properness criterion in Theorem~\ref{thm:proper89}, 
 we have the equivalences.

\begin{tabular}{lll}
 $L$ acts properly on $G/H$
& $\iff$
 $\mu(L) \cap \mu(H)=\{0\}$
\\
&$\iff$
 $n$ is even
 or $n-m \ge 2$.  
\\
\end{tabular}

\vskip 3pc
More generally, 
 one may ask the following question:
\begin{question}
Given a homomorphism 
 $\varphi \colon SL(2,{\mathbb{R}}) \to SL(n,{\mathbb{R}})$, 
determine whether the induced action of $SL(2,{\mathbb{R}})$
 on $SL(n,{\mathbb{R}})/SL(m,{\mathbb{R}})$, 
 is proper.  
\end{question}

According to the Dynkin--Kostant theory, 
 the set of conjugacy classes $\operatorname{Hom}(SL(2,{\mathbb{R}}), G)/G$
 is finite 
 for any reductive Lie group $G$.  
For $G=SL(n,{\mathbb{R}})$, 
 there exists a one-to-one correspondence 
 between these conjugacy classes 
 and the set ${\mathbb{P}}(n)$ of all partitions of $n$:
\begin{equation}
\label{eqn:DK}
 \operatorname{Hom}(SL(2,{\mathbb{R}}), G)/G \simeq {\mathbb{P}}(n).  
\end{equation}
More explicitly, 
 any homomorphism $\varphi \colon SL(2,{\mathbb{R}}) \to G$
 is conjugate to a direct sum of the form 
\[
   \underset{j=1}{\overset n \bigoplus} (\overbrace{\varphi_j \oplus \cdots \oplus \varphi_j}^{m_j}), 
\]
 where $\varphi_j$ denotes
 the irreducible $j$-dimensional real representation 
 of $SL(2,{\mathbb{R}})$, 
 and $m_j$ $(\in {\mathbb{N}})$ is the (possibly zero) multiplicity, 
 satisfying $\underset{j=1}{\overset n\sum} j m_j =n$.  
Let $L:=\varphi(SL(2,{\mathbb{R}}))$.  
After conjugating $L$
 if necessary, 
 and using the convention of \eqref{eqn:aG}, 
 we obtain  
\[
  {\mathfrak{a}}_L= {\mathbb{R}}
         (\underset{j=1}{\overset n \bigoplus} (\overbrace{v_j \oplus \cdots \oplus v_j}^{m_j})), 
\]
where $v_j:=(j-1, j-3, \dots, 1-j) \in {\mathbb{Z}}^j$.  
Applying the properness criterion
 in Theorem~\ref{thm:proper89} 
 to the pair $(L, H)=(\varphi(SL(2,{\mathbb{R}})), SL(m,{\mathbb{R}}))$, 
 we conclude
 that the action of $SL(2,{\mathbb{R}})$ on $G/H$ is proper
 if and only if:
\[
   \sum_{j:\text{odd}} j m_j < n-m.  
\]

\subsection{Properly Discontinuous Actions of Surface Groups}
~~~\newline
The previous example $G/H=SL(n,{\mathbb{R}})/SL(m,{\mathbb{R}})$ is 
 a non-symmetric homogeneous space.  
When $G/H$ is a {\emph{reductive symmetric space}}, 
 Okuda \cite{O13}
 provided a complete classification of such spaces
 that admit proper actions 
 of $SL(2,{\mathbb{R}})$
 via a homomorphism 
 $\varphi \colon SL(2,{\mathbb{R}})\to G$.  
His classification relies on the properness criterion
 (Theorem~\ref{thm:proper89})
 along with the Dynkin--Kostant theory 
 of nilpotent orbits, 
 as given in \eqref{eqn:DK}.  
Using this classification, 
 he further established the following result:

\begin{theorem}
[Okuda \cite{O13}]
\label{thm:okuda}
Let $G/H$ be a reductive symmetric space.  
Then the five conditions (i)---(v) in Proposition~\ref{prop:surface_gp} are equivalent.  
\end{theorem}

For a pair of real reductive Lie groups
 $G \supset H$
 that does {\emph{not}} form a symmetric pair, 
 the implication (v) $\Rightarrow$ (ii) 
 in Proposition~\ref{prop:surface_gp}
 does not necessarily hold.

\medskip

\subsection{Solvable Case}
~~~\newline
So far, 
 we have primarily discussed proper actions
 (or properly discontinuous actions)
 on homogeneous spaces $G/H$
 in the setting 
 where $G$ is a reductive Lie group.  
In contrast, 
 when $G/H$ is a simply connected {\emph{solvable}} homogeneous space, 
 the Calabi--Markus phenomenon does {\emph{not}} occur.  
In fact, the following theorem is based on a structural result
 on solvable Lie groups 
 due to Chevalley
 \cite{Ch41}.

\begin{theorem}
[{\cite[Thm.\ 2.2]{Ko93}}]
Suppose that $G$ is a solvable Lie group
 and $H$ is a proper closed subgroup of $G$.  
Then there exists a discrete subgroup $\Gamma$ of $G$
 that acts properly discontinuously and freely on $G/H$,  
 such that the fundamental group $\pi_1(\Gamma \backslash G/H)$
 is infinite.  
\end{theorem}

\section{Cocompact Discontinuous Groups}
\label{sec:cocompact}

One of the central and challenging problems
 concerning discontinuous groups
 acting on non-Riemannian homogeneous spaces $G/H$ is the following:

\begin{problem}
[{\cite[Problem B]{K01}}]
\label{prob:G1}
Determine all pairs $(G,H)$
for which $G/H$ admits {\emph{cocompact}} discontinuous groups.  
\end{problem}

Problem~\ref{prob:G1} is a long-standing open problem, 
 and it remains unsolved even
 when $G/H$ is a symmetric space
 of rank one, 
 as exemplified by the space form conjecture
 (Conjecture \ref{conj:G4}).

{}From now on, 
 we focus on the case
 where $G$ is a real reductive linear Lie group
 and $H$ is a reductive subgroup.  
We recall, 
 as seen in Proposition~\ref{prop:redGH}, 
 that the homogeneous space $G/H$ admits 
 a pseudo-Riemannian structure
 with respect to which $G$ acts as a group of isometries.

In the classical case where $H$ is compact, 
 a theorem of Borel \cite{Bo62} affirms Problem~\ref{prob:G1}
 by establishing the existence
 of cocompact arithmetic discrete subgroups in $G$.

When $G$ is noncompact, 
 cocompact discontinuous groups for $G/H$
 are much smaller
 than cocompact lattices in $G$.  
For example, 
 their cohomological dimensions
 are strictly smaller \cite[Cor.\ 5.5]{K89}.  
A simple approach to Problem~\ref{prob:G1} is
 to consider a \lq\lq{continuous analog}\rq\rq\
 of discontinuous groups $\Gamma$, 
 thereby leading to the notion 
 of the {\emph{standard quotient}}, 
 as described below.

\subsection{Standard Quotient $\Gamma \backslash G/H$}
~~~\par
We continue to work under the standing assumption
 that $G$ is a real reductive linear Lie group 
 and that $H \subset G$ is a reductive subgroup.

\begin{definition}
[Standard Quotient {\cite[Def.\ 1.4]{KasK16}}]
\label{def:standard}
Suppose $L$ is a reductive subgroup of $G$
 such that the action of $L$ on $G/H$ is proper.  
Then any torsion-free discrete subgroup $\Gamma$
 of $L$ is a discontinuous group for $G/H$;
 that is, 
 the $\Gamma$-action on $G/H$ is properly discontinuous and free.  
The quotient space $\Gamma \backslash G/H$
 is referred to as a {\emph{standard quotient}} of $G/H$.  
\end{definition}

The properness criterion stated in Theorem~\ref{thm:proper89} 
 provides a convenient method 
 for checking whether a given reductive subgroup $L \subset G$
 satisfies the condition required in Definition~\ref{def:standard}.

\subsection{Finding Cocompact Discontinuous Groups}
~~~\par
If a subgroup $L$ as in Definition~\ref{def:standard} acts cocompactly on $G/H$, 
 then $G/H$ admits a cocompact discontinuous group $\Gamma$, 
 obtained by taking $\Gamma$ to be a torsion-free cocompact discrete subgroup of $L$, 
 where existence is guaranteed by Borel's theorem.  
A necessary and sufficient condition for such a subgroup $L$, 
 which acts properly on $G/H$, 
 to act cocompactly is 
\begin{equation}
\label{eqn:K89_Thm47}
   d(L) + d (H) = d(G), 
\end{equation}
as established in \cite[Thm.\ 4.7]{K89}, 
where $d(G):= \dim {\mathfrak{p}}= \dim G/K$.

A list of reductive homogeneous spaces $G/H$
 that admit proper and cocompact actions
 of reductive subgroups may be found in \cite{KnK25}, 
 summarizing earlier lists 
 including \cite{K89, K97}.  
A particularly important subclass
 consists of irreducible symmetric spaces, 
 which are the main focus of \cite{KY05}.  
These works, 
 in particular, 
 provide examples of compact pseudo-Riemannian locally homogeneous spaces
 $\Gamma \backslash G/H$
 realized as standard quotients of $G/H$.

The following conjecture was proposed 
by the author in \cite{K01}.  

\begin{conjecture}
[{\cite[Conj.\ 4.3]{K01}}]
\label{conj:G1}
The homogeneous space $G/H$ of reductive type admits
 a cocompact properly discontinuous group 
 if and only if $G/H$ admits a compact standard quotient.  
\end{conjecture}

If Conjecture~\ref{conj:G1} were proved to be true, 
 then Problem~\ref{prob:G1} would reduce
 to the following one:

\begin{problem}
\label{prob:G3}
Classify all pairs $(G,H)$
 such that $G/H$ admits a compact standard quotient.  
\end{problem}

This problem is expected to be tractable, 
 as it reduces to checking
 a finite number of representation-theoretic conditions
 for each $G/H$
 in order to verify the properness criterion and the cocompactness criterion in \cite[Thms 4.1 and 4.7]{K89}.

Tojo \cite{Tj} showed that the list of irreducible symmetric spaces $G/H$
in \cite{KY05} admitting proper and cocompact actions
 of reductive subgroups $L$
 is, 
in fact, 
 complete  up to compact factors in the case
 where $G$ is a simple Lie group.  
This result provides a solution to Problem~\ref{prob:G3}
 in the case
 where $G$ is a symmetric space
 with $G$ simple.

Furthermore, 
 Boche{\'n}ski \cite{Bo22} studied the case
 where $G$ is the direct product of two absolutely simple groups.  
A more recent preprint of Boche{\'n}ski--Tralle \cite{BoT24} shows
 that, under the assumption
 that $G$ is absolutely simple, 
 the list in \cite{KY05} contain
 all the homogeneous spaces $G/H$
 that admit proper and cocompact actions
 of reductive subgroups $L$, 
 up to compact factors and switching $L$ and $H$, 
 thereby yielding further progress on Problem~\ref{prob:G3}.  

\begin{remark}
(1)\enspace
Conjecture~\ref{conj:G1} does not assert
 that all cocompact discontinuous groups are standard.  
Indeed, 
 there exist reductive homogeneous spaces $G/H$
 that admit {\emph{non-standard}} compact quotients;
 that is, 
 there exist triples $(G, H, \Gamma)$
 such that $\Gamma$ is a cocompact discontinuous group for $G/H$, 
 while the Zariski closure of $\Gamma$
 fails to act properly on $G/H$;
 see \cite{K98, Kas12, KnK25}.  
\par\noindent
(2)\enspace
An analogue of Conjecture~\ref{conj:G1} was established
 by Okuda \cite{O13}
 for semisimple symmetric spaces $G/H$.  
It is worth noting that this result replaces the key assumption 
 of cocompactness 
 in the original conjecture
 with the requirement 
 that $\Gamma$ is a surface group $\pi_1(\Sigma_g)$,
 as stated in Theorem \ref{thm:okuda}.  
\end{remark}

Special cases of Conjecture~\ref{conj:G1} include the following:

\begin{conjecture}
[Space Form Conjecture {\cite[Conj.~2.6]{K01}}]
\label{conj:G4}
There exists a compact, complete, 
 pseudo-Riemannian manifold
 of signature $(p,q)$
 with constant sectional curvature $1$
 if and only if $(p,q)$ lies
 in the following list:

\begin{figure}[H]
\begin{center}
\begin{tabular}{cccccc}
\hline
$p$
& ${\mathbb{N}}$
& 0
& 1
& 3
& 7
\\
\hline
$q$
& 0
& ${\mathbb{N}}$
& $2{\mathbb{N}}$
& $4{\mathbb{N}}$
& 8
\\
\hline
\end{tabular}
\end{center}
\end{figure}
See also Example~\ref{ex:GammaXpq} (6) below
 for the {\emph{tangential analogue}} 
 in the context of Cartan motion groups.  
\end{conjecture}

\begin{conjecture}
\label{conj:SLSL}
For any non-trivial homomorphism $\psi \colon SL(m,{\mathbb{R}}) \to SL(n,{\mathbb{R}})$ with $m<n$, 
 the homogeneous space $SL(n,{\mathbb{R}})/\psi(SL(m,{\mathbb{R}}))$ does not admit 
 a cocompact discontinuous group.  
\end{conjecture}

The following are notable special cases
 of Conjecture~\ref{conj:SLSL}, 
 corresponding to specific choices of $\psi$. 
 For a concise overview of these methods, including a discussion of their limitations and applications, see \cite{K01, KT-pre}.

\begin{example}
\label{ex:4.7}
{\rm{(1)}}\enspace
($\psi$ is the identity map.)
For the standard representation $\psi$, 
 the homogeneous space 
 $SL(n,{\mathbb{R}})/\psi(SL(m,{\mathbb{R}}))$
 does not admit a cocompact discontinuous group.  
\newline
{\rm{(2)}}\enspace
($\psi$ is an irreducible representation.)
For any {\emph{irreducible}} representation
 $\psi \colon SL(2,{\mathbb{R}}) \to SL(n,{\mathbb{R}})$ with $n \ge 5$, 
 the homogeneous space $SL(n,{\mathbb{R}})/\psi(SL(2,{\mathbb{R}}))$
 does not admit a cocompact discontinuous group.  
\end{example}

The first statement in Example~\ref{ex:4.7} has been studied
 over 35 years 
 with affirmative results
 obtained for \lq\lq{generic parameters}\rq\rq.  
A complete solution 
 was recently announced
 by Kassel, Morita, 
 and Tholozan \cite{KT-pre, KMT-pre}.  
Earlier contributions include
 \cite{KICM90sate, KDuke92, Z94, LZ95, LMZ95, B96, S00, Th, M17}, 
 which employed a variety of approaches from different areas.

The second statement in Example~\ref{ex:4.7} was proved
 by Margulis (\cite{Ma97})
 based on the notion of {\emph{tempered subgroups}}, 
 defined by the asymptotic behaviour
 of matrix coefficients
 of unitary representations
 under the restriction from $G$ to its subgroup $H$, 
 symbolically written as $G \downarrow H$.  
In contrast to this notion, 
 we will explore the notion 
of \emph{tempered homogeneous spaces}
 $G/H$
 in Section \ref{sec:quantify}
 on the regular unitary representation on $L^2(G/H)$, 
 symbolically written as $H \uparrow G$
 (see Definition \ref{def:temp_space}).

Whereas the idea of standard quotients $\Gamma \backslash G/H$
 is to replace a discrete subgroup $\Gamma$ 
with a connected subgroup $L$
 (Definition~\ref{def:standard}), 
 one may instead consider an \lq\lq{approximation}\rq\rq\
 of Problem~\ref{prob:G1}, 
 by replacing the homogeneous space
 $X=G/H$ with the {\emph{tangential homogeneous space}}
\[
  X_{\theta}:=G_{\theta}/H_{\theta}, 
\]
 where $G_{\theta} := K \ltimes {\mathfrak{p}}$
 is the Cartan motion group 
 associated with the real reductive group $G=K \exp {\mathfrak{p}}$
 and similarly for $H_{\theta}$.  
If $G/H$ admits a compact standard quotient, 
 then the tangential homogeneous space
 $G_{\theta}/H_{\theta}$
 admits a cocompact discontinuous group.  
The group $G_{\theta}$ is a compact extension
 of the abelian group ${\mathfrak{p}}$, 
 and has a much simpler structure.

\vskip 1pc
We consider the following {\emph{tangential}} question
 related to Problem~\ref{prob:G1}:

\begin{problem}
[{\cite{KY05}}]
\label{prob:G5}
For which pairs $(G,H)$ of real reductive Lie groups, 
 does the tangential homogeneous space $G_{\theta}/H_{\theta}$
 admit a cocompact
 discontinuous group?
\end{problem}

This problem is expected to be significantly simpler
 than the original one,
 yet it remains unsolved
 even in the case of symmetric spaces.  
Nevertheless, 
 a complete answer is available for {\emph{tangential}} pseudo-Riemannian space forms, 
 using a theorem of Adams \cite{A62}
 on the maximal number
 of pointwise linearly independent continuous vector fields
 on spheres;
 see Example~\ref{ex:GammaXpq} (6) below.

At the end of this section, 
 we briefly review these problems
 and conjectures, 
 taking the pseudo-Riemannian space form $X(p,q)_+$
 as a representative example.  
We also highlight recent developments in the field
 (see, {\it{e.g.}}, \cite{KConj23, KT-pre, KnK25} and references therein).

\begin{example}
\label{ex:GammaXpq}
Let $(G,H)=(O(p+1,q), O(p,q))$,  
 and let $X=X(p,q)_+=G/H$ denote the pseudo-Riemannian space form
 of signature $(p,q)$ 
 as in Example~\ref{ex:Rpq_surface}.  
\par\noindent
(1) ({\cite{CM, Ku, K89}})\enspace
$X(p,q)$ admits a discontinuous group
 of infinite order if and only if $p<q$.  
\par\noindent
(2) ({\cite{Ku, O13}})\enspace
$X(p,q)$ admits a discontinuous group 
 isomorphic to a surface group
 if and only if $p+1 <q$
 or $p+1=q \in 2 {\mathbb{N}}$.  
\par\noindent
(3) ({\cite{Ku, KO90, Th, mor19}})\enspace
If $X(p,q)$ admits a cocompact discontinuous group, 
 then $p q=0$ or
 $p<q$ with $q \in 2{\mathbb{N}}$.  
\par\noindent
(4) ({\cite{Ku, K01}})\enspace
$X(p,q)$ admits a cocompact discontinuous group
 if $(p,q)$ is in the list, 
 as stated in Conjecture~\ref{conj:G4}.  

\par\noindent
(5) ({\cite{Kas12, KnK25}})\enspace
If $(p,q)=(0,2)$, $(1,2)$, or $(3,4)$, 
 then $X(p,q)$ admits a cocompact discontinuous group
 that can be continuously deformed
 into a Zariski dense subgroup of $G$, 
 while preserving proper discontinuity of the action.  
Moreover, 
 for $(p,q)=(1,2n)$ $(n \ge 2)$, 
 the anti-de Sitter space $X(1,2n)$ admits a compact quotient
 which has a non-trivial continuous deformation
 within the class of standard quotients.  
\par\noindent
(6) ({\cite{KY05}})\enspace
The tangential homogeneous space
 $G_{\theta}/H_{\theta}$ admits a cocompact discontinuous group
 if and only if 
 $p< \rho(q)$ 
 where $\rho(q)$ is the {\emph{Radon--Hurwitz number}}.  
Equivalently, 
 this condition holds if and only if $(p,q)$ appears in the following list:

\begin{figure}[H]
\begin{center}
\begin{tabular}{ccccccccccccccc}
\hline
$p$
& ${\mathbb{N}}$
& 0
& 1
& 2
& 3
& 4
& 5
& 6
& 7
& 8
& 9 
& 10
& 11
& $\cdots$
\\
\hline
$q$
& 0
& ${\mathbb{N}}$
& $2{\mathbb{N}}$
& $2{\mathbb{N}}$
& $4{\mathbb{N}}$
& $8{\mathbb{N}}$
& $8{\mathbb{N}}$
& $8{\mathbb{N}}$
& $8{\mathbb{N}}$
& $16{\mathbb{N}}$
& $32{\mathbb{N}}$
& $64{\mathbb{N}}$
& $64{\mathbb{N}}$
& $\cdots$
\\
\hline
\end{tabular}
\end{center}
\end{figure}

\end{example}

\section{Proper Maps and Unitary Representation}
\label{sec:proper_unitary}

This section explores the relationship
 between the properness of group actions 
 and representation theory, 
 particularly
 in the context of discretely decomposable unitary representations.

\subsection{Compact-Like Actions and Compact-Like Unitary Representations}
\label{subsec:cpt_like}
~~~\par
Every continuous action of a compact group is proper
 (see Definition-Lemma~\ref{def:properaction}).  
In this sense, 
 a proper action may be viewed
 as a {\emph{compact-like action}}.

Every unitary representation of a compact group
 decomposes discretely into a direct sum
 of irreducible representations.  
Thus, 
 discretely decomposable unitary representations 
 may be viewed as {\emph{compact-like representations}}.

A proposal to connect two seemingly different areas---{\emph{proper actions}}
 in topology and {\emph{discrete decomposability}}
 in representation theory---by observing how non-compact subgroups
 can exhibit compact-like behaviour 
 within infinite-dimensional automorphism groups
 was first articulated in the 2000 paper 
 \cite[Sect.\ 3]{K-Hayashibara97}.

In this section, 
 we review the foundational concepts
 and give an overview of some developments
 in this direction
 since then.

\medskip
\subsection{Discrete Decomposable Unitary Representations}
~~~\par
Let $\widehat G$ denote the {\emph{unitary dual}} 
 of a locally compact group $G$;
 that is, 
 the set of equivalence classes 
 of irreducible unitary representations of $G$, 
 endowed with the Fell topology.

By a theorem of Mautner \cite[Chap.\ VIII, Sect.\ 41]{Mt50}, 
 every unitary representation of the group
 decomposes into a direct integral 
 of irreducible unitary representations.

Let $G'$ be a subgroup of $G$.  
Suppose that $\pi \in \widehat G$.  
Then, 
 the restriction $\pi|_{G'}$, 
 as a unitary representation of the subgroup $G'$, 
 can be decomposed 
 into a direct integral of irreducible unitary representations:
\begin{equation}
\label{eqn:direct_int}
\pi|_{G'} \simeq \int_{\widehat{G'}}^{\oplus} m_{\pi}(\tau)  \tau d \mu(\tau), 
\end{equation}
where $\mu$ is a Borel measure 
 on the unitary dual $\widehat{G'}$, 
 and
\[
   m_{\pi} \colon \widehat{G'} \to {\mathbb{N}} \cup \{\infty\}
\]
is  a measurable function called the {\emph{multiplicity function}}
 for the direct integral \eqref{eqn:direct_int}.  
This irreducible decomposition is known as {\emph{the branching law}}.  
Typically, 
 it involves continuous spectrum
 when $G'$ is non-compact.

The concept of {\emph{$G'$-admissible restrictions}} was
 introduced in \cite[Sect.\ 1]{KInvent94} 
 in a general setting
 that includes the case
 where $G'$ is a {\emph{non-compact}} subgroup.

\begin{definition}
\label{def:deco}
The restriction $\Pi|_{G'}$ is said to be {\emph{$G'$-admissible}}
 if it can be decomposed discretely 
into a direct sum of irreducible unitary representations
 $\pi$ of $G'$:
\[\text
{
$\Pi|_{G'} \simeq {\underset{\pi \in \widehat{G'}}\sum}^{\oplus} m_{\pi}\pi$
\quad({discrete sum})
}
\]
where the multiplicity $m_{\pi}:=[\Pi|_{G'}:\pi]$ is finite 
 for every $\pi \in \widehat{G'}$.  
\end{definition}

We refer to \cite{KInvent94, K98b}
 for the criterion of $G'$-admissibility
 for the restriction of an irreducible unitary representation
 of a reductive Lie group $G$ to its reductive subgroup $G'$.  
See also Kitagawa \cite{Ki25} for some recent developments.

Discretely decomposable restrictions may be regarded
 as {\emph{compact-like representations}}.  
We examine the discrete decomposability
 of representations from two perspectives:
one based on the properness
 of the moment map (Section~\ref{subsec:coadjoint}), 
 and the other based on {\emph{proper actions}} of groups
 (Section~\ref{subsec:K17} and Theorem~\ref{thm:HG_rest}).  


\subsection{Coadjoint Orbits and Proper Maps}
\label{subsec:coadjoint}
~~~\par
Let $G$ be a Lie group, 
 and ${\mathfrak{g}}^{\ast}$ the dual of the Lie algebra ${\mathfrak{g}}$.  
The {\emph{orbit method}} initiated by Kirillov, 
 and developed by Kostant, Duflo, 
 and Vogan among others,
 is a philosophy
 that seeks to understand the unitary dual $\widehat G$
 through the coadjoint representation
 $\operatorname{Ad}^{\ast}
 \colon G \to GL_{\mathbb{R}}({\mathfrak{g}}^{\ast})$.

For $\lambda \in {\mathfrak{g}}^{\ast}$, 
 ${\mathcal{O}}_{\lambda}:=\operatorname{Ad}^{\ast}(G) \lambda$ is called
 a {\emph{coadjoint orbit}}.  
The quotient space ${\mathfrak{g}}^{\ast}/\operatorname{Ad}^{\ast}(G)$
 parametrizes coadjoint orbits.  
Loosely speaking, 
 the orbit method suggests the existence
 of a \lq\lq{natural correspondence}\rq\rq\
 between a subset of the set ${\mathfrak{g}}^{\ast}/\operatorname{Ad}^{\ast}(G)$
 and the unitary dual $\widehat G$.  
Indeed, 
 there exists a natural bijection 
\begin{equation}
\label{eqn:geom_Q}
   Q \colon {\mathfrak{g}}^{\ast}/\operatorname{Ad}^{\ast}(G) \overset \sim \to \widehat G
\end{equation}
when $G$ is a simply connected nilpotent Lie group, 
 as Kirillov established in his 1962
 celebrated paper \cite{Ki62}.  
For reductive Lie groups $G$, 
 there is no such natural bijection as in \eqref{eqn:geom_Q}, 
 however, 
 one still expects that the orbit method provides 
 insight into unitary representations
 of $G$
 via a deep relationship 
between ${\mathfrak{g}}^{\ast}/\operatorname{Ad}^{\ast}(G)$ and $\widehat G$.

For any $\lambda \in {\mathfrak{g}}^{\ast}$, 
 the skew-symmetric bilinear map 
\[
  \lambda \colon {\mathfrak{g}} \times {\mathfrak{g}} \to {\mathbb{R}}, \quad
  [X,Y] \mapsto \lambda([X,Y])
\]
induces a $G$-invariant symplectic form
 on the coadjoint orbit 
\[
   {\mathcal{O}}_{\lambda}:=\operatorname{Ad}^{\ast}(G) \lambda \simeq G/G_{\lambda}, 
\]
 which is known as the {\emph{Kostant--Kirillov--Souriau symplectic form}}.  
The momentum map of the $G$-action on ${\mathcal{O}}_{\lambda}$
 is precisely the canonical injection
 ${\mathcal{O}}_{\lambda} \hookrightarrow {\mathfrak{g}}^{\ast}$, 
 and hence ${\mathcal{O}}_{\lambda}$ is a $G$-Hamiltonian manifold.

{}From this perspective, 
 if one can associate an irreducible unitary representation 
 $\Pi_{\lambda}:=Q({\mathcal{O}}_{\lambda})$
 naturally to a coadjoint orbit ${\mathcal{O}}_{\lambda}$, 
 we may regard $\Pi_{\lambda}$ 
 as a geometric quantization of ${\mathcal{O}}_{\lambda}$.

Let $H$ be a subgroup of $G$.  
By the branching problem we aim to understand
 the restriction $\Pi|_H$
 of a representation $\Pi$ of $G$
 to the subgroup $H$.  
Suppose that $\Pi$ is an irreducible unitary representation
 that corresponds to a coadjoint orbit ${\mathcal{O}}$
 in ${\mathfrak{g}}^{\ast}$.  
We observe
 that the canonical projection
\[
   \operatorname{pr} \colon {\mathfrak{g}}^{\ast} \to {\mathfrak{h}}^{\ast}
\]
 for the dual of the Lie algebras 
 ${\mathfrak{h}} \hookrightarrow {\mathfrak{g}}$
 is equivariant 
 with respect to the coadjoint action of $H$.  
In the spirit of the orbit method, 
 the restriction $\Pi|_H$ might be 
 interpreted 
 in terms of the image $\operatorname{pr}({\mathcal{O}})$
 as a union of $H$-coadjoint orbits, 
 suggesting how the restriction $\Pi|_H$ decomposes under $H$.

The expected correspondences may be illustrated as follows:
\begin{align*}
\text{unitary dual}\quad\widehat G \ni \Pi 
&\overset{\text{orbit method}}
{\leftarrow \cdots \rightarrow}
{\mathcal{O}}\subset {\mathfrak{g}}^{\ast}
\quad
\text{(coadjoint orbit)}
\\
\text{subgroup}\quad H \subset G 
&\hphantom{M}{\leftarrow \cdots \rightarrow}
\hphantom{M}\operatorname{pr} \colon {\mathfrak{g}}^{\ast} \to {\mathfrak{h}}^{\ast}
\quad
\text{projection}
\\
\text{$\pi|_H$ is $H$-admissible${}^{\ast}$}
&\hphantom{M}{\leftarrow\overset{?}\cdots\rightarrow}
\hphantom{M}\text{${\mathcal{O}}_{\pi} \hookrightarrow {\mathfrak{g}}^{\ast} 
\overset{\operatorname{pr}}{\to} {\mathfrak{h}}^{\ast}$ is proper.  }
\end{align*}
Here, 
 the question mark indicates
 a conjectural equivalence
 between $H$-admissibility
 of the restriction $\Pi|_H$ 
 and the properness of the moment map on the coadjoint orbit ${\mathcal{O}}$.

\begin{question}
\label{q:deco_orbit}
Suppose that $\Pi \in \widehat G$ 
 is a \lq\lq{geometric quantization}\rq\rq\
 of a coadjoint orbit ${\mathcal{O}} \subset {\mathfrak{g}}^{\ast}$
 in the sense of the orbit method.  
Let $H$ be a reductive subgroup of $G$.  
Is the following equivalence (i) $\Leftrightarrow$ (ii) valid?
\newline
(i)\enspace
The restriction $\Pi|_H$ is $H$-admissible.  
\newline
(ii)\enspace
The projection $\operatorname{pr} \colon {\mathfrak{g}}^{\ast} \to {\mathfrak{h}}^{\ast}$
 is a proper map
 when restricted to the coadjoint orbit ${\mathcal{O}}$.  
\end{question}

See \cite{DV10, KN03, xKN, P15}
 for some affirmative cases
 and related discussions.

Although beyond the scope of this article, 
 we note that for non-reductive subgroups,
 Duflo introduced the notions
 of a {\emph{weakly proper map}}, 
 which relaxes the properness condition
 appearing in condition (ii)
 of Question~\ref{q:deco_orbit}.  
See \cite{LOY23} for example.

\subsection{
Proper Action and Discrete Decomposability
}
\label{subsec:K17}
~~~\newline
We recall some basic notions from the theory 
 of infinite-dimensional representations of Lie groups, 
 not necessarily unitary.  
Let $\Pi$ be a continuous representation
 of a Lie group $G$
 on a complete, locally convex topological vector space $V$
 ({\it{e.g.}}, a Banach space), 
 and let $V^{\infty}$ denote the space
 of smooth vectors.  
Then $V^{\infty}$ is dense in $V$, 
 and carries a natural topology.  
The representation $\Pi$ induces a continuous representation $\Pi^{\infty}$
 on $V^{\infty}$, 
 and a dual representation $\Pi^{-\infty}$ on the continuous dual space
 $V^{-\infty}$ of $V^{\infty}$.

Now suppose that $G$ is a real reductive Lie group.  
Let ${\mathcal{M}}(G)$ denote the category 
 of smooth admissible representations
 of finite length with moderate growth, 
 which are defined on Fr{\'e}chet topological vector spaces \cite[Chap.\ 11]{W92}.  
Let $\operatorname{I r r}(G)$ denote 
 the set of irreducible objects in ${\mathcal{M}}(G)$.

While we do not go into the precise definition
 of the category ${\mathcal{M}}(G)$ here, 
 it is helpful to keep 
 in mind
 that $\operatorname{Irr}(G)$ contains
 the smooth representations $\Pi^{\infty}$
 of irreducible unitary representations $\Pi$ of $G$.  
This gives a natural injection:
\begin{equation}
\label{eqn:smooth_unitary}
\widehat G \hookrightarrow \operatorname{Irr}(G), \quad
  \Pi \mapsto \Pi^{\infty}.  
\end{equation}

Let $H$ be a closed subgroup of a Lie group $G$.  

\begin{definition}
\label{def:distinguished}
We say that $\Pi\in \operatorname{Irr}(G)$
 is an {\emph{$H$-distinguished representation}}
 of $G$, 
 if {$(\Pi^{-\infty})^H \ne \{0\}$}, 
 or equivalently, 
 by Frobenius reciprocity, 
\begin{equation*}
\operatorname{Hom}_G(\Pi,\, C^{\infty}(G/H)) \ne \{0\}.  
\end{equation*}
Let $\operatorname{Irr}(G)_H$ denote 
 the subset of $\operatorname{Irr}(G)$
 consisting
 of $H$-distinguished irreducible admissible representations, 
 and let $\widehat G_H := \widehat G \cap \operatorname{Irr}(G)_H$
 via the injection given in \eqref{eqn:smooth_unitary}.  
\end{definition}

In line with the philosophy of {\emph{compact-like actions}}
 discussed in Section~\ref{subsec:cpt_like}, 
 which links geometry
 and function spaces, 
 the following four properties are closely related:
\par
$\bullet$\enspace
the action of $G'$ on $X$ is proper;
\par
$\bullet$\enspace
the action of $G'$ on $X$ is {\textit{compact-like}};
\par
$\bullet$\enspace
the representation of $G'$ on $C^{\infty}(X)$ is {\textit{compact-like}};
\par
$\bullet$\enspace
for any $\Pi \in \operatorname{Irr}(G)$
 occurring in $C^{\infty}(X)$, 
 the restriction $\Pi|_{G'}$ is
\newline
\hphantom{MM}
 discretely decomposable.  
\newline
This philosophy holds under the additional assumption 
 of sphericity, 
 as formalised in Theorem~\ref{thm:HG_rest} below.

We briefly recall the definition of sphericity:
\begin{definition}
Let $X_{\mathbb{C}}$ be a connected complex manifold
 on which a complex reductive Lie group $G_{\mathbb{C}}$
 acts holomorphically.  
The action of $G_{\mathbb{C}}$ is said to be {\emph{spherical}} 
 if a Borel subgroup of $G_{\mathbb{C}}$ has an open orbit
 in $X_{\mathbb{C}}$.  
\end{definition}

\begin{example}
{\rm{(1)}}\enspace
The complexification $G_{\mathbb{C}}/H_{\mathbb{C}}$
 of a reductive symmetric space $G/H$ is spherical.  
\newline
{\rm{(2)}}\enspace
Any flag variety is spherical.  
\end{example}

\begin{theorem}
[{\cite{K17}}]
\label{thm:HG_rest}
Let $X=G/H$ be a reductive symmetric space.  
Suppose that $G'$ is a reductive subgroup of $G$, 
 and that its complexification $G_{\mathbb{C}}'$ acts spherically on $X_{\mathbb{C}}$.  
If the action of $G'$ on $X$ is proper, 
 then any irreducible $H$-distinguished unitary representation $\Pi$ of $G$
 is $G'$-admissible;
 in particular, 
 it decomposes discretely upon restriction to the subgroup $G'$.  
Moreover, the multiplicities are uniformly bounded:
\[
   \underset{\Pi \in \operatorname{Irr}(G)_H}\sup\,\,
   \underset{\pi \in \operatorname{Irr}(G')}\sup\,\,
   [\Pi|_{G'}:\pi]<\infty.  
\]
\end{theorem}

We now give three examples
 to illustrate this result.  
\begin{example}
[Standard Anti-de Sitter Manifolds]
Let $X$ be an odd-dimensional 
 anti-de Sitter space, 
 {\it{i.e.}}, 
\[
   X=G/H=SO(2n,2)/SO(2n,1).  
\] 
The subgroup $G':=U(n,1)$ acts properly on $X$, 
 and its complexification
 $G_{\mathbb{C}}'=G L(n+1, {\mathbb{C}})$ acts spherically
 on the complex manifold $X_{\mathbb{C}}=SO(2n+2, {\mathbb{C}})/SO(2n+1, {\mathbb{C}})$, 
 which is biholomorphic to the $(2n+1)$-dimensional complex sphere $S_{\mathbb{C}}^{2n+1}$.  
Therefore, 
 Theorem~\ref{thm:HG_rest} applies in this case.  
The corresponding discretely decomposable branching laws
 are explicitly obtained in \cite[Thm.\ 6.1]{KInvent94}.  
\end{example}

\begin{example}
[Pseudo-Riemannian Space Form of Signature $(8,7)$]
Consider a 15-dimensional manifold
 by 
\[
   X=G/H=SO(8,8)/SO(8,7).  
\]
Then $X$ is a pseudo-Riemannian manifold of signature $(8, 7)$, 
 with constant negative sectional curvature.  
The subgroup $G'={\mathrm{Spin}}(1,8)$ of $G$ acts properly on $X$, 
 and its complexification $G_{\mathbb{C}}'= {\mathrm{Spin}}(8, {\mathbb{C}})$
 acts spherically on $X_{\mathbb{C}}\simeq S_{\mathbb{C}}^{15}$.  
Thus, 
 Theorem~\ref{thm:HG_rest} applies here.  
The corresponding discretely decomposable branching laws
 for the restriction $SO(8,8) \downarrow Spin(1,8)$
 are explicitly obtained
 in \cite[Thm. 5.5]{K17} and \cite{STV18}.  
\end{example}

\begin{example}
[Indefinite K{\"a}hler Manifolds]
The homogeneous space 
\[
X=G/H=SO(2n,2)/U(n,1)
\]
 admits a natural indefinite K{\"a}hler structure.  
The subgroup $G'=SO(2n,1)$ acts properly on $X$, 
 and the complexified group
 $G_{\mathbb{C}}'=S O(2n+1, {\mathbb{C}})$ acts spherically
 on the complex manifold $X_{\mathbb{C}}= S O(2n+2, {\mathbb{C}})/G L(n+1,{\mathbb{C}})$.  
Hence, 
 Theorem~\ref{thm:HG_rest} applies here.  
A detailed account of the geometric setting and the discretely decomposable branching laws
 for the restriction $G \downarrow G'$ can be found
 in \cite[Sect.\ 6]{K09}, 
specifically for the case $n=2$.  
\end{example}

In the setting of Theorem~\ref{thm:HG_rest}, 
 let $X_{\Gamma} =\Gamma \backslash G/H$ be
 a standard locally symmetric space
  (Definition~\ref{def:standard}), 
 where $\Gamma$ is a torsion-free discrete subgroup of $G'$.  
Equipped with the pseudo-Riemannian structure inherited from 
 the symmetric space $X=G/H$, 
 the space $X_{\Gamma}$ provides a natural framework 
 for spectral analysis.  
In fact,   
Theorem~\ref{thm:HG_rest} serves 
 as a cornerstone for the analytic theory on standard locally symmetric spaces
 $X_{\Gamma}$, 
 as developed in the monograph \cite{KaK25}.

\section{Two Quantifications of Proper Actions}
\label{sec:quantify}

Two notions that may appear unrelated at first glance---originating
 respectively from joint works with Kassel \cite{KasK16}
 and Benoist \cite{BeKoI, BeKoII}---in fact arose from distinct
 and independent motivations. 
In this section, however, we reinterpret them from a unified perspective:
 as two approaches to {\emph{quantifying the properness}} of group actions.

$\bullet$\enspace The notion of {\emph{sharpness}} provides a means
 of measuring how strongly a given action satisfies the condition of properness
 (see Section~\ref{subsec:sharp}).

$\bullet$\enspace 
The other, 
 based on dynamical volume estimates, 
 quantifies the extension
 to which an action fails to be proper
 (see Sections~\ref{subsec:measure}---\ref{subsec:Lp}).

\subsection{Sharp Action}
\label{subsec:sharp}
~~~\newline
As a strengthening of the properness condition
 for group actions,
 we recall the notion of {\emph{sharpness}}, 
 introduced in \cite{KasK16}.

Let $G$ be a linear reductive Lie group.  
Let 
\[
   \mu \colon G \to \overline{{\mathfrak{a}}_+}
\]
denote the Cartan projection 
 associated with the Cartan decomposition $G=K \overline{A_+} K$, 
 as defined in \eqref{eqn:Cartanpr}.

Let $H$ be a closed subgroup, 
 and let $X:=G/H$ be the associated homogeneous space.

\begin{definition}
[Strongly Proper Action: Sharpness Constants]
\label{def:sharp1}
Let $\Gamma$ be a discrete subgroup of $G$.  
We say that $\Gamma$ is {\emph{sharp}} for $X$
if there exist constants $c \in (0,1]$
 and $C \ge 0$
 such that
\[
\|\mu(\gamma)-\mu(H)\| \ge c \|\mu(\gamma)\|-C
\]
holds for all $\gamma \in \Gamma$.  
In this case, 
 the quotient space $X_{\Gamma}:=\Gamma \backslash G/H$
 is called a {\emph{sharp quotient}} of $X$.

The constants $(c, C)$ are called the sharpness constants. 
\end{definition}

This notion can be reformulated
 in terms of the {\emph{asymptotic cone}} 
(also known as the {\emph{limit cone}}), 
 which we recall now.

Let $V$ be a finite-dimensional vector space over ${\mathbb{R}}$, 
 and let $S$ be a subset of $V$.  
\begin{definition}
[Asymptotic Cone, Limit Cone]
\label{def:as_cone}
The {\emph{asymptotic cone}} of $S$, 
 also referred to as the {\emph{limit cone}}, 
 is a closed cone in $V$ 
 consisting of all limit points of sequences of the form
\[
\underset{n \to \infty} \lim \varepsilon_n x_n,
\]
 where $x_n \in S$
 and $\varepsilon_n >0$
 is a sequence converging to 0.  
We denote this cone by $S \infty$.  
\end{definition}

The following lemma is an immediate consequence
 of Definition~\ref{def:pitch} of the relation $\pitchfork$.

\begin{lemma}
\label{lem:asym_tube}
Let $S$ and $T$ be subsets of the vector space in $V$.  
If the asymptotic cones satisfy $S \infty \cap T \infty =\{0\}$, 
 then $S \pitchfork T$ in $V$, 
 where $V$ is regarded as an additive group.  
\end{lemma}

We next restate Definition~\ref{def:sharp1}
in an equivalent form.  

\begin{definition}
[Sharp Action]
\label{def:sharp2}
Let $\Gamma$ be a discrete subgroup of $G$.  
The action of $\Gamma$ on $X$ is called {\emph{sharp}}
 if 
\[
   \mu(\Gamma) \infty \cap \mu(H) \infty =\{0\}.  
\]
\end{definition}

If the action of $\Gamma$ on $X=G/H$ is sharp, 
 then it follows from Lemma~\ref{lem:asym_tube}
 that 
\[
   \text{$\mu(\Gamma) \pitchfork \mu(H)$ in ${\mathfrak{a}}$.}  
\]
Hence, 
 by the properness criterion, 
 as stated in Theorem~\ref{thm:proper96}, 
 the $\Gamma$-action on $X$ is proper.

The converse implication
\[
   \text{proper action} \Rightarrow \text{sharp action}
\]
does not hold in general.  
However, 
 there are many interesting examples
 in which sharpness does follow:

--- When $H$ is reductive, 
 any standard quotient (Definition~\ref{def:standard}) $X_{\Gamma}$ is sharp.  

--- Remarkably, 
 Kassel and Tholozan have announced in a recent preprint \cite{KT-pre}
 an affirmative solution
 to the Sharpness Conjecture \cite[Conj.\ 4.12]{KasK16}, 
 which asserts that any {\emph{cocompact}} discontinuous group for $G/H$
 is sharp.

\vskip 1pc
An advantage of the notion of sharpness
 is that it becomes particularly effective
 in the study of {\emph{deformations}}
 of discontinuous groups.

In contrast to the Selberg--Weil rigidity theorem 
 for the Riemannian symmetric space $G/K$, 
 irreducible pseudo-Riemannian symmetric spaces
 may admit cocompact discontinuous groups
 that are {\emph{not locally rigid}}, 
 even in arbitrarily high dimensions.  
This phenomenon was first observed in the early 90s
 (see \cite[Remarks 2 and 3]{Ko93})
 for the group manifold $G$, 
 viewed as a homogeneous space $(G \times G)/\operatorname{diag} G$.

A major difficulty in studying deformations
 of discontinuous groups lies in the fact
 that,
 when $H$ is noncompact, 
 small deformations of a discrete subgroup
 can easily destroy the properness of the action.  
In the context of 3-dimensional compact anti-de Sitter manifolds, 
 Goldman \cite{G85} conjectured
 that any small deformation 
 of a standard cocompact discontinuous group 
 preserves proper discontinuity.  
This conjecture
 was proved by the present author
 \cite{K98}, 
 based on the properness criterion, 
 as stated in Theorem~\ref{thm:proper96}.

The idea introduced in \cite{K98}, 
 further developed by Kassel \cite{Kas12} 
 and related works, 
 exploits the fact
 that the limit cone $\mu(\Gamma)\infty$
 remains well-controlled
 under small deformations of $\Gamma$.  
Consequently,  
 proper discontinuity
 is maintained through small deformation---under a mild condition---provided 
 that the initial group is a {\emph{sharp}}
 discontinuous group.

The notion of sharpness also plays a significant role 
 in other problems, 
 such as the orbit counting problem
 for properly discontinuous actions of $\Gamma$
 on pseudo-Riemannian symmetric spaces $X$.  
This is exemplified in the construction of the {\emph{stable spectrum}}
 for $\Gamma \backslash X$
 in \cite{KasK16}.  
On the other hand, 
 sharpness also proves useful 
 in addressing the existence problem
 of cocompact discontinuous groups, 
 as seen in \cite{KT-pre}.


\subsection{Measure-Theoretic Approach to Proper Actions}
\label{subsec:measure}
~~~\par
Whereas the previous section discussed
 the notion of sharp actions
 as a quantitative strengthening of properness, 
 the present section takes
 the opposite perspective:
 it introduces a quantitative method
 to measure the extent to which a group action fails
 to be proper.

We begin with a reformulation
 of the definition of proper actions
 (Definition-Lemma~\ref{def:properaction}), 
 using {\emph{measure-theoretic conditions}}
 in lieu
 of the original topological definition.

Let $G$ be a locally compact group, 
 and let $X$ be a locally compact space
 equipped with a continuous $G$-action.  
Suppose further that $X$ carries a Radon measure $\mu$.  
Then, 
 for every compact subset $S \subset X$, 
 the function
\[
   G \to {\mathbb{R}}, \quad
   g \mapsto \operatorname{vol}(S \cap g S):= \mu(S \cap g S)
\]
 is continuous with respect to the topology on $G$.  

\centerline{%
\raisebox{-3em}{\includegraphics[scale=.22]{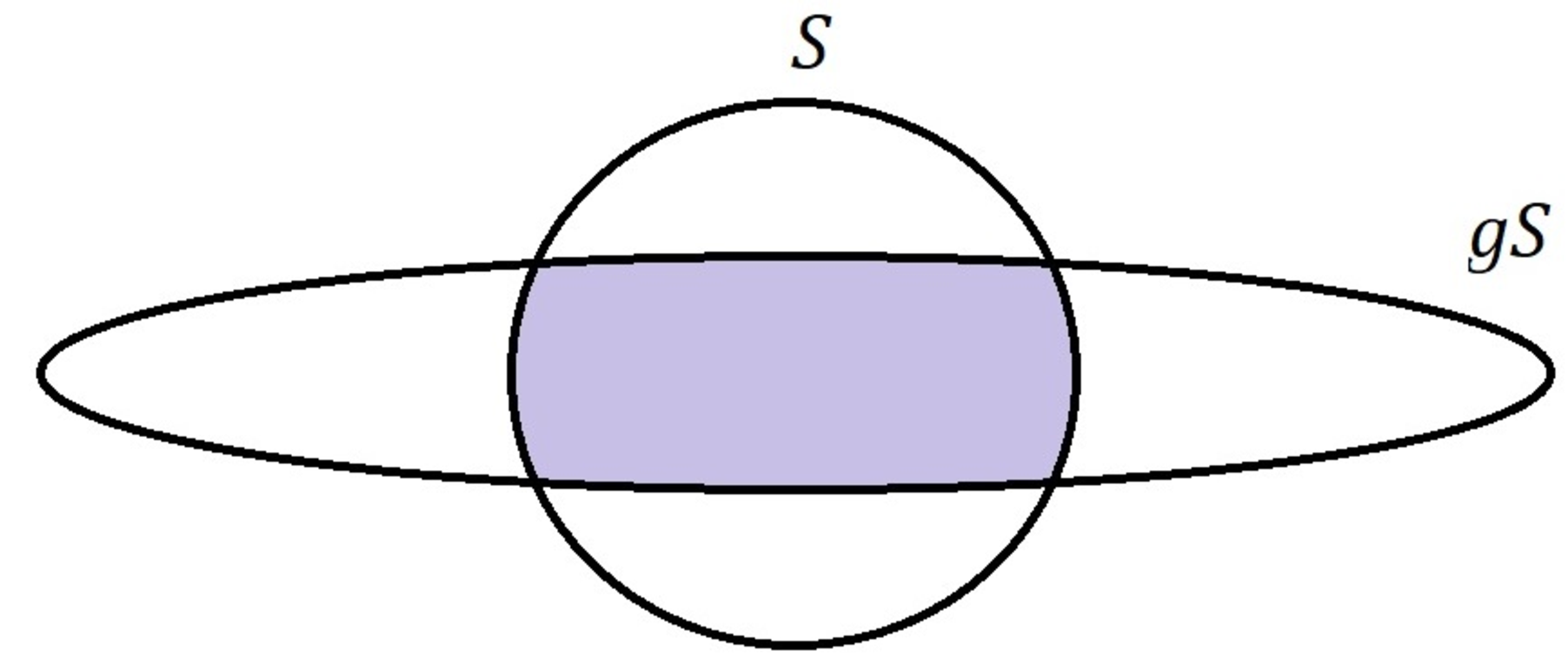}}
}

\begin{lemma}
\label{lem:doi_250427}
The following two conditions are equivalent:
\newline
{\rm{(i)}}\enspace
The action of $G$ on $X$ is proper;
\newline
{\rm{(ii)}}\enspace
For every compact subset $S \subset X$, 
 the function $\operatorname{vol}(S \cap g S)$ has compact support
 on $G$.  
\end{lemma}
\begin{proof}
(i) $\Rightarrow$ (ii).  
The function
 $g \mapsto \operatorname{vol}(S \cap g S)$ is continuous, 
 and its support is contained in 
\[
   \operatorname{Supp} \operatorname{vol}(S \cap g S) 
   \subset 
   \{g \in G: S \cap g S \ne \emptyset\}=:G_S.  
\]
Hence, 
 if the $G$-action on $X$ is proper, 
 ({\it{i.e.,}} $G_S$ is compact for all compact $S \subset X$), 
 then the function has compact support.  
\newline
(ii) $\Rightarrow$ (i). 
Conversely,
 suppose that the action of $G$ on $X$ is not proper.  
Then there exists a compact subset $S \subset X$
 such that 
 $G_S$ is not compact.

Choose an open,
 relatively compact subset $V \subset X$
 with $S \subset V$, 
 and let $S'$ be the closure of $V$, 
 which is compact.  
For each $g \in G_S$, 
 we have 
\[
  \emptyset \ne S \cap g S \subset V \cap g V \subset S' \cap g S'.  
\]
Since $S \cap g S$ is open 
 and
 has positive measure
 (as $\mu$ is a Radon measure), 
 it follows that $\mu(V \cap g V)>0$.  
Hence, 
\[
\operatorname{Supp}(\operatorname{vol}(S' \cap g S'))
 \supset G_S, 
\]
 which is not compact.  
Thus, 
 by contraposition, 
 (ii) implies (i).  
\end{proof}

We now focus on the case
 when the action is {\emph{not}} proper.

By the preceding lemma, 
 there exists a compact subset $S \subset X$
 such that the volume function
\[
   g \mapsto \operatorname{vol}(S \cap g S)
\]
does not have compact support.

To quantitatively assess
 the degree of non-properness quantitatively, 
 we examine how this function behaves
 at the \lq\lq{infinity}\rq\rq\
 in $G$.

We may expect that the action of $G$ on $X$ is 
 {\emph{close to being proper}}
 if the volume function $\operatorname{vol}(g S \cap S)$ decays 
 \lq\lq{rapidly as $g \in G$ tends to infinity}\rq\rq.

\vskip 1pc
\subsection
{An Example of Volume Estimate: $\operatorname{vol}(S \cap g S)$}
~~~\newline
To illustrate this principle, 
 consider a simple yet instructive example
 showing the asymptotic behavior
 of $\operatorname{vol}(S \cap g S)$.

Let 
\( G := {\mathbb{R}}\)~ act
 on ${\mathbb{R}}^2\setminus \{(0,0)\}$
 by
\[
   (x,y) \mapsto (e^t x, e^{-t}y), 
   \quad
   \text{where $t\in {\mathbb{R}}$}.  
\]
As observed in Example~\ref{ex:R2a}, 
 this action is {\emph{free}}
 and {\emph{all orbits are closed}}, 
 but it is not {\emph{proper}}.  
In particular, 
 the $G$-action on the entire space $X:={\mathbb{R}}^2$ is not proper.  
{}From a measure-theoretic point of view, 
 $X$ and $X \setminus \{(0,0)\}$
 are equivalent.  
To understand failure of properness quantitatively, 
 consider the asymptotic behavior of the function
\[
  t \mapsto \operatorname{vol}(S \cap t \cdot S), 
\]
 where the translate of a compact subset $S \subset G$ by $t\in {\mathbb{R}}$
 is defined by:
\[
  t \cdot S:=\{(e^t x, e^{-t} y):(x,y) \in S\}.  
\]

\begin{claim}
\label{lem:rec_area}
If the origin $o=(0,0)$ is an interior point of a compact subset $S \subset {\mathbb{R}}^2$, 
 then there exist constants $C_1$, $C_2 >0$
 such that 
\[
  C_1 e^{-|t|} 
 \le \operatorname{vol}(S \cap t \cdot S)
 \le C_2 e^{-|t|}
\]
for all $t \in {\mathbb{R}}$.  
\end{claim}

\begin{proof}
We begin with the case 
 where $S$ is the square
\[
   D_R:=\{ (x,y) \in {\mathbb{R}}^2:|x| \le R, |y| \le R\}.
   \hskip3.5em\raisebox{-1.5em}{\smash{\includegraphics[scale=0.12]{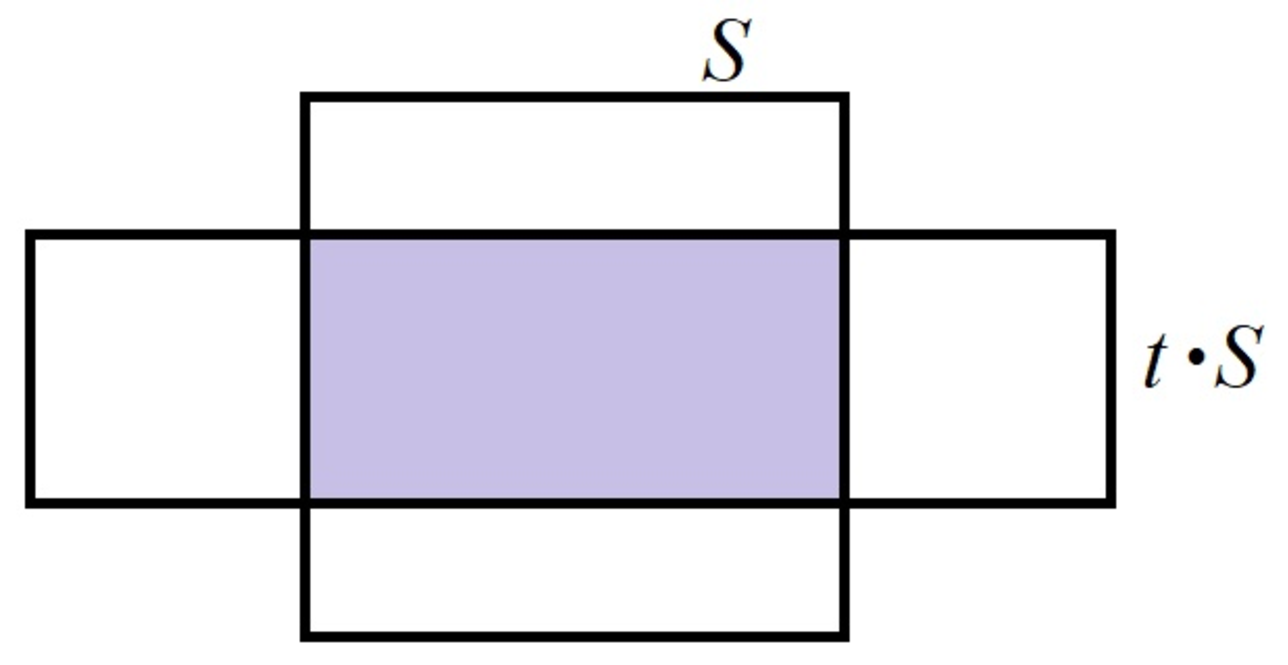}}}
\]
A direct computation shows that 
\[
\operatorname{vol}(t \cdot D_R \cap D_R) = 4R^2 e^{-|t|} = \operatorname{vol}(D_R) e^{-|t|}.
\]

Now, 
 suppose $S$ is a compact set containing the origin
 as an interior point.  
Then there exist constants $0 < r < R$
 such that $D_r \subset S \subset D_R$.  
It follows that 
\[
  \operatorname{vol}(D_r \cap t \cdot D_r)
  \le 
  \operatorname{vol}(S \cap t \cdot S)
  \le 
  \operatorname{vol}(D_R \cap t \cdot D_R).  
\]
Using the earlier formula, 
\[
  \operatorname{vol}(D_r) e^{-|t|}
  \le 
  \operatorname{vol}(t \cdot S \cap S)
  \le 
  \operatorname{vol}(D_R) e^{-|t|}.  
\]
This completes the proof.  
\end{proof}

\subsection
{Function $\rho_V$ and Constant $p_V$}
~~~\par
The previous example extends naturally to higher dimensions.  
Before formulating this generalization, 
 we recall the function $\rho_V$,
 which is associated with a finite-dimensional representation
 of a Lie algebra on a vector space, 
 introduced in \cite{BeKoI, BeKoII}.

Let ${\mathfrak{h}}$ be a Lie algebra, 
 and suppose
\[
 \tau \colon {\mathfrak{h}} \to \operatorname{End}_{\mathbb{R}}(V)
\]
 is a representation of ${\mathfrak{h}}$
 on a finite-dimensional real vector space $V$.  
We define a non-negative function
\begin{equation}
\label{eqn:rhoV}
 \rho_V \colon {\mathfrak{h}} \to {\mathbb{R}}_{\ge 0}
\end{equation}
as follows:
For each $Y \in {\mathfrak{h}}$, 
 let $\{\lambda_1, \dots, \lambda_n\}$ be the multiset
 of generalized eigenvalue, 
 of $\tau(Y)$, 
 viewed as a complex-linear operator
 on $V_{\mathbb{C}} := V \otimes_{\mathbb{R}} {\mathbb{C}}$.  
Then define
\[
   \rho_V(Y):= \frac 1 2 \sum_{j=1}^n |\operatorname{Re}\lambda_j|.  
\]

Now assume that ${\mathfrak{h}}$ is an algebraic Lie algebra, 
 and that $\tau \colon {\mathfrak{h}} \to \operatorname{End}_{\mathbb{R}}(V)$
 is an algebraic representation.  
Let ${\mathfrak{a}} \subset {\mathfrak{h}}$ be a maximally split abelian subalgebra.  
Then the function $\rho_V$ is entirely determined 
 by its restriction to ${\mathfrak{a}}$, 
 and we have
\[
\rho_V(Y)=\sum_{j=1}^{n} |\lambda_j(Y)|\quad \text{for $Y \in {\mathfrak{a}}$}
\]
since $\tau(Y) \in \operatorname{End}_{\mathbb{R}}(V)$
 is diagonalizable in this case.

\begin{example}
Let ${\mathfrak{h}}$ be a semisimple Lie algebra.  
For the adjoint representation 
\[
  \operatorname{ad} \colon 
 {\mathfrak {h}} \to \operatorname{End}({\mathfrak{h}}),
\]
 the function $\rho_{\operatorname{ad}}$ coincides
 with twice the \lq\lq{usual $\rho$-function}\rq\rq\ 
 on the positive Weyl chamber $\overline{\mathfrak{a}_+}$;
 that is,
\[
   \rho_{\operatorname{ad}}(Y)=2 \rho(Y)
   = \sum_{\alpha \in \Sigma^+({\mathfrak{g}}, {\mathfrak{a}})} \alpha(Y)
\quad
\text{for $Y \in \overline{{\mathfrak{a}}_+}$}.  
\]
It is worth noting
 that while the \lq\lq{usual $\rho$-function}\rq\rq\ is linear, 
 our function $\rho_{\operatorname{ad}}$ is only {\emph{piecewise linear}}.  
\end{example}

We now consider the ratio between two $\rho$-functions:
\par\hphantom{M}
$\bullet$\enspace
one associated with a given representation $(\tau, V)$, 
 and
\par\hphantom{M}
$\bullet$\enspace
the other with the adjoint representation.

\begin{definition}
[The Invariant $p_V$]
\label{def:pV}
Let $(\tau, V)$ be an algebraic representation of ${\mathfrak{h}}$.  
We define the invariant $p_V$ by 
\[
   p_V:= \underset{Y \in {\mathfrak{h}} \setminus \{0\}}{\max}
\frac{\rho_{\mathfrak{h}}(Y)}
     {\rho_V(Y)}.  
\]
If ${\mathfrak{a}} \subset {\mathfrak{h}}$ is a maximally split abelian subalgebra, 
 this simplifies to
\[
p_V=
\underset{Y \in {\mathfrak{a}} \setminus \{0\}}{\max}
\frac{\rho_{\mathfrak{h}}(Y)}
     {\rho_V(Y)}.
\]
In terms of eigenvalues, 
 this becomes
\[
p_V
=
\underset{Y \in {\mathfrak{a}} \setminus \{0\}}{\max}
\frac{\sum |\text{eigenvalues of $\operatorname{ad}(Y) \in \operatorname{End}({\mathfrak{h}})$}|}
{\sum |\text{eigenvalues of $\tau(Y) \in \operatorname{End}(V)$}|}.
\]
\end{definition}

\begin{example}
\label{ex:pV_sl2}
Consider the standard representation
 of ${\mathfrak{h}}={\mathfrak{sl}}(2,{\mathbb{R}})$
 on $V={\mathbb{R}}^2$.  
Let 
\[
   {\mathfrak{a}}={\mathbb{R}}H, \quad
   H:=\begin{pmatrix} 1 & 0 \\ 0 & -1\end{pmatrix}.  
\]
A straightforward computation shows
\begin{align*}
  \rho_V(t H)=& \frac 1 2 (|t|+|-t|)=|t|, 
\\
  \rho_{\operatorname{ad}}(t H)=& \frac 1 2 (|2t|+0+|-2t|)=2|t|, 
\end{align*}
Therefore, 
 the invariant $p_V$ is: 
\[
  p_V=\underset{t \ne 0} \max \frac{2 |t|}{|t|} =2.  
\]
\end{example}

Before explaining the meaning of the invariant $p_V$, 
 we introduce the concept
 of an {\emph{almost $L^p$ function}},  
 which helps clarify the broad picture.

\medskip
\subsection{Almost $L^p$ Function}
~~~\par
Let $Z$ be a locally compact space
 equipped with a Radon measure.  

\begin{definition}
[Almost $L^p$ Function]
\label{def:Lpfn}
A measurable function $f$ is said to be 
 {\emph{almost}} $L^p$
 if 
\[
   f \in \bigcap_{\varepsilon >0} L^{p+\varepsilon} (Z).  
\]
\end{definition}

\begin{example}
Let $D$ be the unit disk, equipped with the Poincar{\'e} metric
\[
  ds^2=\frac{4(d x^2+ d y^2)}{(1-x^2-y^2)^2}, 
\]
and let $\Delta$ be the corresponding Laplace-Beltrami operator.  
We define the function $p(\lambda)$ by 
\[
   p(\lambda)
  =
  \begin{cases}
  \frac{2}{1-\sqrt{1-4\lambda}}
  \quad
  &\text{for $0 \le \lambda \le \frac 1 4$}, 
\\
  2
  &\text{for $\frac 1 4 \le \lambda$}.  
  \end{cases}
\]
Suppose that $f \in C^{\infty}(D)$ is an eigenfunction 
 of $\Delta$, 
 satisfying:
\[
   \Delta f = \lambda f
\]
for some $\lambda \ge 0$, 
 and suppose further that $f$ is $SO(2)$-finite.  
Then $f$ is almost $L^{p(\lambda)}$. 
Here, 
 a smooth function $f$ is said to be $SO(2)$-{\emph{finite}}
 if the complex vector space $\operatorname{span}_{\mathbb{C}}\{f(k(x,y)):k \in SO(2)\}$ is finite-dimensional.   
\end{example}

If $p \le p'$, 
 then clearly:
\[
  \text{$f$ is almost $L^p$}
  \Rightarrow
  \text{$f$ is almost $L^{p'}$.  }
\]

Hence, 
 if a function $f$ is almost $L^p$
 for some exponent $p$, 
 then there exists a {\emph{minimal}} (or {\emph{optimal}}) exponent $q \le p$
 such that $f$ is almost $L^q$, 
 in the sense
 that 
\begin{align*}
q=\,&\text{$\inf\{p'>0:\text{$f \in L^{p'+\varepsilon}(Z)$ for all $\varepsilon>0$}\}$}
\\
=\,&\text{$\min\{p'>0: f \in L^{p'+\varepsilon}(Z)\text{ for all $\varepsilon>0$}\}$}.  
\end{align*}

\medskip 
\subsection{Optimal $L^p$-Exponent $q(G;X)$}
~~~\par
Suppose that a unimodular, locally compact group $G$ acts continuously
 on a locally compact space $X$, 
 equipped with a Radon measure.  
We now introduce the following invariant 
 associated with this group action,
 denoted by $q(G;X)$, 
 which measures the optimal decay rate
 of the volume function.  
\begin{definition}
[Optimal $L^p$-Exponent $q(G;X)$]
\label{def:qGX}
The invariant $q(G;X)$ is defined 
 to be the optimal constant $q>0$
 such that, for every compact subset $S \subset X$, 
 the function
\[
  \text{$g \mapsto \operatorname{vol}(S \cap g S)$}
\]
 is an almost $L^q$ function on $G$.  
In other words, 
\[
    \operatorname{vol}(S \cap g S) 
    \in \bigcap_{\varepsilon >0} L^{q+\varepsilon}(G).  
\]
\end{definition}

A general question is the following:
\begin{problem}
\label{prob:qGX}
Find an explicit formula of $q(G;X)$
 in terms of geometric 
 or representation-theoretic data associated with the action of $G$ on $X$.  
\end{problem}

\begin{example}
[$q(G;V)$ for $G=SL(2, {\mathbb{R}})$ acting
 on $V={\mathbb{R}}^2$]
\label{ex:SL2qX}
Consider the standard action of $G=SL(2, {\mathbb{R}})$
 on $V={\mathbb{R}}^2$.  
Then $q(G;V)=2$.  
Let us explain
 why this holds.

Recall the Cartan decomposition $G=KAK$, 
 with $g=k(\theta_1) a(t)k(\theta_2)$, 
 where 
\[
   \text{$a(t)=\begin{pmatrix} e^t & 0 \\ 0 & e^{-t}\end{pmatrix}$
 and $k(\theta)=\begin{pmatrix} \cos \theta & -\sin \theta \\ \sin \theta & \cos \theta\end{pmatrix}$}.  
\]
As seen in Example~\ref{ex:R2b}, 
 the action of $A$ on ${\mathbb{R}}^2$ is given by 
\[
(x,y) \mapsto (e^tx, e^{-t}y).  
\]
Now we take $S \subset {\mathbb{R}}^2$
 to be a $K$-invariant compact subset
 ({\it{i.e.,}} $S$ is rotationally invariant), 
 and observe that under the Cartan decomposition $G=KAK$, 
 the volume function satisfies
\[
   \operatorname{vol}(S \cap g S)
   =
   \operatorname{vol}(S \cap k(\theta_1) a(t) k(\theta_2) S)
   =
   \operatorname{vol}(S \cap a(t) S)
   \le 
   Ce^{-|t|},   
\]
by Claim~\ref{lem:rec_area}.

The Haar measure on $G=SL(2,{\mathbb{R}})$, 
 expressed via  the Cartan decomposition, 
 is given by
\[
   \sinh (2t) d t d \theta_1 d\theta_2.  
\]
Therefore, 
 the function $\operatorname{vol}(S \cap gS)$ belongs to $L^p(G)$
 for any $p >2$.  
Since:
\[
  \text{$\int_{\mathbb{R}} e^{-p t} \sinh(2t) d t < \infty$
        if and only if $p >2$.}  
\]
Thus, 
 we have $q(X) \le 2$.  
Conversely, 
 by Claim~\ref{lem:rec_area} again, 
 there exists a compact subset $S \subset {\mathbb{R}}^2$
 such that the opposite inequality also holds:
\[
   C' e^{-|t|} \le \operatorname{vol}(S \cap g S)
\]
for some constant $C'>0$, 
 which shows
 that $q(X) \ge 2$.  
Hence, 
 we conclude
 that $q(G;V)=2$ if $(G, V)=(SL(2, {\mathbb{R}}), {\mathbb{R}}^2)$.  
\end{example}

We observe that the value obtained in the above example coincides with $p_V=2$ from Example \ref{ex:pV_sl2}.  
This is not a mere coincidence; rather,
it reflects a more general principle,
as reflected in Proposition~\ref{prop:baby} below.

Indeed, 
 Example~\ref{ex:SL2qX} extends naturally to any faithful representation
 of a reductive group.  
This generalization elucidates the relationship
 between the algebraic invariant $p_V$,  defined in Definition~\ref{def:pV}, 
 and the optimal constant $q(G;X)$
 (see Definition~\ref{def:qGX})
 for which 
\[
   \operatorname{vol}(S \cap h S) \in L^{p+\varepsilon}(G)
\]
 for all $\varepsilon >0$, 
 when the action of $G$ on $X$
 is linear, 
 as follows.

\begin{proposition}
[{\cite{BeKoI}}]
\label{prop:baby}
Suppose that $G$ is a real reductive linear group.  
Let $\tau \colon G \to SL_{\pm}(V)$ be a finite-dimensional representation 
 on a real vector space $V$ 
 with compact kernel.  
Then the following equality holds:
\[
   p_V=q(G;V).  
\]
\end{proposition}

\begin{proof}
[Sketch of Proof]
Let $G=KAK$ be a Cartan decomposition.  
For $g =k_1 e^Y k_2$, 
 and for a compact subset $S \subset V$
 containing 0 as an interior point, 
 one has
\[
\operatorname{vol}(S \cap g S) 
\sim 
e^{-\rho_V(Y)}, 
\]
 as stated in Claim~\ref{lem:rec_area}.

Asymptotically, 
 the Haar measure $d g$ on $G$ satisfies
\[
d g \sim e^{\rho_{\mathfrak{h}}(Y)} d k_1 d Y d k_2
\quad
\text{(away from wall)}.
\]
This leads to the proof of Proposition~\ref{prop:baby}.  
\end{proof}

\subsection{Tempered $G$-Spaces}
\label{subsec:Lp}
~~~\par
We recall the notion 
 of tempered unitary representations
 of a locally compact group $G$.  

\begin{definition}
[Tempered Unitary Representation]
A unitary representation $\pi$ is said
 to be {\emph{tempered}}
 if $\pi \prec L^2(G)$;
 that is, 
 if $\pi$ is weakly contained
 in the regular representation 
 on $L^2(G)$.  
\end{definition}

Suppose that $X$ is a locally compact space
 equipped with a Radon measure $\mu$,
 on which a locally compact group $G$ acts 
 continuously 
 and in a measure-preserving manner.  
Then there is a natural unitary representation of $G$
 on the Hilbert space
 ${\mathcal{H}}=L^2(X, \mu)$.

We note that the assumption of a $G$-invariant measure 
 can be dropped.  
Nevertheless, 
 one can still define a canonical unitary representation---{\emph{regular representation}}---of $G$
 on the Hilbert space of $L^2$-sections
 of the half-density bundle over $X$.

\begin{definition}
[Tempered $G$-Spaces]
\label{def:temp_space}
We say that $X$ is a {\emph{tempered space}}
 if the regular representation of $G$ on $L^2(X)$
 is a tempered unitary representation.  
\end{definition}

A general question is the following:
\begin{problem}
\label{prob:tempered_space}
Given a homogeneous space $G/H$, 
 determine a criterion on the pair $(G,H)$ 
 that ensures $G/H$ is a tempered space.  
\end{problem}

We explain the background of Problem~\ref{prob:qGX}
 in connection with the theory 
 of unitary representations.  
\begin{definition}
[Almost $L^p$-Representation]
\label{def:almostLp}
For $p \ge 1$, 
 a unitary representation $\pi$ of $G$
 on a Hilbert space ${\mathcal{H}}$
 is called {\emph{almost $L^p$}}
 if there is a dense subspace
 $D \subset {\mathcal{H}}$ such that
 the matrix coefficient
\[
\text{$(\pi(g)u, v)_{\mathcal{H}}$
 is an almost $L^p$ function on $G$
\quad
for all $u, v\in D$}.  
\]
\end{definition}

Cowling--Haagerup--Howe \cite{CHH} proved the following.  
\begin{theorem}
\label{thm:CHH}
Let $G$ be a semisimple Lie group.  
Then $\pi$ is tempered
if and only if 
 $\pi$ is almost $L^2$.  
\end{theorem}

It should be noted that an analogous equivalence may fail
 when $G$ is not semisimple.  

\begin{example}
Let $G={\mathbb{R}}$, 
 and let ${\bf{1}}$ denote the trivial one-dimensional unitary representation
 of $G$.  
Then the matrix coefficient is a constant function on $G$, 
 which does not belong to $L^p({\mathbb{R}})$ for any $p \ne \infty$, 
 even though ${\bf{1}} \prec L^2(G)$.  
\end{example}

For a compact subset $S \subset X$, 
 we denote by $\chi_S$ the characteristic function of $S$, 
 defined by 
\[
  \chi_S(x)=
  \begin{cases}
               1 \quad&\text{if $x \in S$,}
\\
               0 &\text{if $x \not \in S$}.  
  \end{cases}
\]
Then the matrix coefficient for $\chi_S$, $\chi_T \in L^2(X)$, 
 associated with compact subsets $S$, $T\subset X$
 is given by
\begin{align*}
  (\pi(g)\chi_S, \chi_T)_{L^2(X)}=&\int_X \chi_S(g^{-1}x) \chi_T(x)d\mu(x)
\\
                        =&\operatorname{vol}(g S \cap T).  
\end{align*}
Thus,
 Proposition~\ref{prop:baby},
 combined with Theorem~\ref{thm:CHH}, 
 yields a solution to Problem \ref{prob:tempered_space}
 in the linear case:

\begin{theorem}
[{\cite{BeKoI}}]
\label{thm:baby}
Suppose that $G$ is a real reductive linear group.  
Let $\tau \colon G \to SL_{\pm}(V)$
 be a finite-dimensional representation
 on a real vector space $V$ 
 with compact kernel.  
Then $L^2(V)$ is tempered 
 if and only if $p_V \ge 2$.  
\end{theorem}

This result can be viewed as a basic case 
 in the broader framework aimed at determining 
 when  the regular representation 
on $L^2(X)$ is tempered,
 for a general $G$-space $X$.  
In a series of papers
 \cite{BeKoI, BeKoIII, BeKoII, BeKoIV}, 
 Benoist and the present author developed
 this perspective in a more general setting, 
 focusing on homogeneous spaces of reductive groups, 
 while uncovering new connections 
 beyond the traditional scope of unitary representation theory.  
These developments lie beyond the scope of this article.

In the spirit of this section, 
 one may interpret this line of thoughts as offering a way
 to {\emph{quantify the strength or failure of properness}}
 of the group action.

\vskip 3pc
\leftline{Toshiyuki Kobayashi}
\leftline{Graduate School of Mathematical Sciences, The University of Tokyo}
\leftline{3-8-1, Komaba, Meguro, Tokyo, 153-8914, Japan;}
\leftline{French-Japanese Laboratory in Mathematics and Its Interactions, }
\leftline{FJ-LMI CNRS, IR-2025, Tokyo, Japan.  }
\end{document}